\documentclass{article}

\usepackage{amsmath,amssymb,amsthm,mathrsfs}
\usepackage[overload]{empheq}
\usepackage{graphicx}
\usepackage[colorlinks=true]{hyperref}
\usepackage{pdfsync}
\usepackage{tikz}
\usetikzlibrary{decorations.pathreplacing}
\usetikzlibrary{patterns}

\newtheorem{theorem}{Theorem}
\newtheorem{proposition}{Proposition}
\newtheorem{lemma}{Lemma}
\newtheorem{corollary}{Corollary}
\theoremstyle{definition}
\theoremstyle{definition}
\theoremstyle{definition}\newtheorem{remark}{Remark}

\newcommand{\vertiii}[1]{{\left\vert\kern-0.15ex\left\vert\kern-0.15ex\left\vert #1
    \right\vert\kern-0.15ex\right\vert\kern-0.15ex\right\vert}}

\def\R{\mathbb{R}}
\def\N{\mathbb{N}}

\def\di{\displaystyle}
\def\;{\,}

\def\C{\mathcal{C}}
\def\AC{\mathcal{AC}}
\def\CC{\mathrm{C}}

\def\t{\tau}

\def\L{\mathrm{L}}

\def\E{\mathrm{E}}

\def\RR{\mathrm{R}}

\def\PPP{\mathscr{P}_\E}
\def\QQQ{\mathscr{Q}_\E}

\def\ZB{{\mathrm{ZB}}}
\def\ZO{{\mathrm{Z}\Omega}}
\def\ZWZ{{\mathrm{ZWZ}}}
\def\ZBWZB{{\mathrm{ZBWZB}}}
\def\ZBWZ{{\mathrm{ZB {W}Z}}}
\def\ZBWZOX{{\mathrm{ZB {W}Z}\Omega\mathrm{X}}}
\def\ZWZOX{{\mathrm{Z {W}Z}\Omega\mathrm{X}}}
\def\WZOX{{\mathrm{WZ}\Omega\mathrm{X}}}
\def\RR{{\mathrm{R}}}
\def\RV{{\mathrm{RV}}}
\def\PP{{\mathrm{P}}}
\def\QQ{{\mathrm{Q}}}
\def\TT{{\mathrm{T}}}
\def\FF{{\mathrm{F}}}
\def\GG{{\mathrm{G}}}
\def\HH{{\mathrm{H}}}
\def\KK{{\mathrm{K}}}
\def\JJ{{\mathrm{J}}}
\def\YY{{\mathrm{Y}}}

\title{Linear-quadratic optimal sampled-data control problems: convergence result and Riccati theory}

\author{Lo\"ic Bourdin\footnote{Universit\'e de Limoges, Institut de recherche XLIM, D\'epartement de Math\'ematiques et d'Informatique. UMR CNRS 7252. Limoges, France (\texttt{loic.bourdin@unilim.fr}).}
\and
Emmanuel Tr\'elat\footnote{Sorbonne Universit\'es, UPMC Univ Paris 06, CNRS UMR 7598, Laboratoire Jacques-Louis Lions, Institut Universitaire de France, F-75005, Paris, France (\texttt{emmanuel.trelat@upmc.fr}).}
}

\date{}

\begin{document}

\maketitle

\begin{abstract}
We consider a general linear control system and a general quadratic cost, where the state evolves continuously in time and the control is sampled, \textit{i.e.}, is piecewise constant over a subdivision of the time interval. This is the framework of a linear-quadratic optimal sampled-data control problem. As a first result, we prove that, as the sampling periods tend to zero, the optimal sampled-data controls converge pointwise to the optimal permanent control. 
Then, we extend the classical Riccati theory to the sampled-data control framework, by developing two different approaches: the first one uses a recently established version of the Pontryagin maximum principle for optimal sampled-data control problems, and the second one uses an adequate version of the dynamic programming principle.
In turn, we obtain a closed-loop expression for optimal sampled-data controls of linear-quadratic problems. 
\end{abstract}

\bigskip


\bigskip


\section{Introduction}\label{section1}
Optimal control theory is concerned with the analysis of controlled dynamical systems, where one aims at steering such a system from a given configuration to some desired target by minimizing some criterion. The Pontryagin Maximum Principle (in short, PMP), established at the end of the 50's for general nonlinear continuous-time dynamics (see \cite{pont}, and see \cite{gamk} for the history of this discovery), is certainly the milestone of the classical optimal control theory. It provides a first-order necessary condition for optimality, by asserting that any optimal trajectory must be the projection of an extremal. The PMP reduces the search of optimal trajectories to a boundary value problem over the set of extremals. Optimal control theory, and in particular the PMP, has an immense field of applications in various domains (see \cite{agrach,Bon-Chy03,trel2,bres,brys,BulloLewis,hest,lee,Schattler,seth,trel} and references therein).

The classical version of the PMP that can be found in \cite{pont} is concerned with optimal \textit{permanent control} problems, that is, when the control can be modified at any instant of time. In many problems, achieving the optimal trajectory requires a permanent modification of the control. However, such a requirement is not conceivable in practice for human beings, even for mechanical or numerical devices. Therefore, \textit{sampled-data controls} or \textit{digital controls}, for which only a finite number of modifications is allowed, are usually considered for engineering issues. The fixed switching times at which (and only at which) the sampled-data controls can be modified are usually called \textit{controlling times} or \textit{sampling times}. In \textit{optimal sampled-data control problems}, the control evolves in discrete time and the state evolves in continuous time. The situation differs from what is usually called \textit{discrete-time optimal control problem}, where both the state and the control evolve in discrete time. Some versions of the PMP for discrete-time optimal control problems can be found in the literature (see, e.g., \cite{bolt,Halkin,holt2,holt,mord,seth}).

Since the 60's, an extensive literature deals with sampled-data (or digital) control systems, as evidenced by numerous references and books (see, e.g., \cite{acker,acker2,astu,fada,Iser,land,nesi,nesi2,raga,tou} and references therein), and it is still an active research topic. A significant part of the literature is concerned with $\mathcal{H}_2$-$\mathcal{H}_\infty$ optimization theory (see, e.g., \cite{bami,chen,gero,khar,mirk,souz,toiv}), but many numerical and theoretical treatments of other optimization criteria are also addressed (see, e.g., \cite{aida,azhm,bini,bini2,imur,imur2}).

In this paper we establish some new related results with, to the best of our knowledge, a novel approach based on a recently established version of the PMP that can be applied to optimal sampled-data control problems. We refer to \cite{bourdin-trelat-pontryagin2} (see also~\cite{bourdin-trelat-pontryagin3}) for this PMP.\footnote{Actually in \cite{bourdin-trelat-pontryagin2} one can find a PMP for general nonlinear optimal nonpermanent control problems settled on {\em time scales}, which unifies and extends continuous-time and discrete-time issues (see also \cite{bourdin-trelat-cauchy,bourdin-trelat-pontryagin1}).} Our main results in this paper, Theorems~\ref{thmmainresult} and~\ref{thmmainresult2}, are based on the application of this PMP to the particular framework of Linear-Quadratic Optimal Control Problems (in short, LQOCPs).

LQOCPs are widely studied in the literature (see, e.g., \cite{brys,kwak,lee}) and used in practice. Indeed, considering a quadratic cost functional is very usual and natural, for instance in order to minimize distances to nominal trajectories in tracking problems and, even if dynamical systems are nonlinear in general, linearized systems are frequently considered, for instance for stabilization issues. 
The application of the PMP of \cite{bourdin-trelat-pontryagin2} to a LQOCP with sampled-data controls yields an optimal sampled-data control expressed as a function of the costate, thus, as an \textit{open-loop control} (see Proposition~\ref{proppmpsample} in Section~\ref{section2}).

Then, two questions naturally arise.
\begin{enumerate}
\item The first question concerns the (pointwise) convergence of optimal sampled-data controls to the optimal permanent control as the distances between consecutive sampling times tend to zero. We give in Theorem~\ref{thmmainresult} (Section~\ref{section2}) a positive answer to this question. At this stage, optimal sampled-data controls are still expressed as open-loop controls.
\item The second issue concerns the expression of optimal sampled-data controls as \textit{feedbacks}, \textit{i.e.}, as \textit{closed-loop controls}. We provide in Theorem~\ref{thmmainresult2} (Section~\ref{section3}) such an expression. Two different proofs of Theorem~\ref{thmmainresult2} are given, based respectively on the PMP of \cite{bourdin-trelat-pontryagin2} (see Section~\ref{secfirstproof}) and on an adequate version of the dynamic programming principle (see Section~\ref{secsecondproof}):

\begin{enumerate}
\item For continuous LQOCPs with permanent controls, it is well known how to pass from the open-loop optimal (permanent) control coming from the classical PMP to a closed-loop form expressed in terms of the Riccati matrix, at the price of solving the so-called Riccati matrix differential equation (see \cite{brys,kalm,lee,Schattler,seth,trel}). In Section~\ref{secfirstproof} the proof of Theorem~\ref{thmmainresult2} is derived from the open-loop form (Proposition~\ref{proppmpsample}) resulting from the PMP of~\cite{bourdin-trelat-pontryagin2}. 
\item It is well known that optimal controls of discrete LQOCPs can be expressed as closed-loop controls using the dynamic programming principle (see \cite{bell}), and moreover can be explicitly computed in a recursive way, requiring in particular to solve an explicit discrete Riccati matrix differences equation. Considering state-transition matrices, a general (continuous) LQOCP with sampled-data controls can be written as a discrete LQOCP (see, e.g., \cite[p.~445]{kwak}). This is the point of view adopted in~\cite{bini,bini2} where the authors solve explicitly LQOCPs with sampled-data controls, in a recursive way, for autonomous and homogeneous problems. The proof of Theorem~\ref{thmmainresult2} that we give in Section~\ref{secsecondproof} actually provides an extension of the strategy proposed in~\cite{bini} to the nonautonomous and nonhomogeneous case. Our method is based on the dynamic programming principle but it is not required to write the LQOCP with sampled-data controls as a discrete LQOCP.
\end{enumerate}
These two different approaches lead, in accordance, to Theorem~\ref{thmmainresult2}, and thus complete the Riccati theory for LQOCPs with sampled-data controls. Moreover, as in \cite{bini}, it can be derived from Theorem~\ref{thmmainresult2} a recursive way to compute explicitly the optimal sampled-data controls of LQOCPs (see Corollary~\ref{corollary} in Section~\ref{section3}). 
\end{enumerate}
In turn, combining Theorem~\ref{thmmainresult} and Corollary~\ref{thmmainresult}, we provide in this paper a strategy in order to compute explicitly pointwise convergent approximations of optimal permanent controls of LQOCPs. We provide in Section~\ref{section4} some illustrating numerical simulations.

\section{Pointwise convergence of optimal sampled-data controls}\label{section2}
We first introduce some notations available throughout the paper. Let $m$ and $n$ be two nonzero integers. We denote by $\langle \cdot , \cdot \rangle_m$ (resp., $\langle \cdot , \cdot \rangle_n$) and $\Vert \cdot \Vert_m$ (resp., $\Vert \cdot \Vert_n$) the usual scalar product and the usual Euclidean norm of $\R^m$ (resp., $\R^n$).

Let $a < b$ be two real numbers. We denote by $\C := \C ([a,b],\R^n)$ (resp., $\AC := \AC ([a,b],\R^n)$) the classical space of continuous functions (resp., absolutely continuous functions). We endow $\C$ with its usual uniform norm $\Vert \cdot \Vert_{\infty}$. The convergence of a sequence $(q_k)_{k \in \N}$ to some $q$ in $\C$ (for the corresponding usual strong topology of $\C$) will be denoted by $q_k \to q$.

We denote by $\L^2 : = \L^2 ([a,b],\R^m)$ the classical Lebesgue space of square-integrable functions, endowed with its usual norm $\Vert \cdot \Vert_{\L^2}$. The strong convergence (resp., weak convergence) of a sequence $(u_k)_{k \in \N}$ to some $u$ in $\L^2$ will be denoted by $u_k \to u$ (resp., $ u_k \rightharpoonup u$).

We denote by $\vertiii{\cdot}$ the induced norm for matrices in $\R^{n,n}$, $\R^{n,m}$, $\R^{m,n}$ and $\R^{m,m}$. If $M=M(\cdot)$ is a continuous matrix defined on $[a,b]$, we denote by $\vertiii{M}_\infty$ its uniform norm. Finally, we denote by $M^\top$ the transpose of a matrix $M$.

\begin{remark}\label{rempositivematrix2}
If a square matrix $M$ is positive-semidefinite, then $M_0^\top M M_0$ is positive-semidefinite as well for any matrix $M_0$ having as many rows as $M$.
\end{remark}

\subsection{Preliminaries on LQOCPs}\label{secprelim}
Let $\E$ be a non-empty subset of $\L^2$. We consider the general LQOCP 
\begin{equation}\label{theproblem}\tag{$\PPP$}
\begin{array}{ll}
\text{minimize} & \overline{\CC}(q,u), \\
& \\
\text{subject to} & q \in \AC, \quad u \in \E \subset \L^2, \\[5pt]
& \dot q(t) = A(t) q(t) + B(t) u(t)+\omega(t), \quad \text{a.e. $t \in [a,b]$,}  \\[3pt]
& q(a)=q_a ,
\end{array}
\end{equation}
with
\begin{multline*}
\overline{\CC}(q,u) := \dfrac{1}{2} \Big\langle S \big( q(b)-q_b \big) , q(b)-q_b \Big\rangle_n \\[5pt]
+ \dfrac{1}{2} \di \int_a^{b} \Big\langle W(\t) \big( q(\t)-x(\t) \big) , q(\t) -x(\t) \Big\rangle_n + \Big\langle R(\t) \big( u(\t)-v(\t) \big) , u(\t)-v(\t) \Big\rangle_m \; d\t ,
\end{multline*}
where $q_a$, $q_b \in \R^n$, $S \in \R^{n,n}$, $x$, $\omega : [a,b] \to \R^n$ and $v : [a,b] \to \R^m$ are continuous functions, and $A : [a,b] \rightarrow \R^{n,n}$, $B : [a,b] \rightarrow \R^{n,m}$, $W : [a,b] \rightarrow \R^{n,n}$ and $R : [a,b] \rightarrow \R^{m,m}$ are continuous matrices (see Remark~\ref{assumptionsweakened} for weakened regularity assumptions).

We assume that $S$ is positive-semidefinite and that $W(t)$ and $R(t)$ are respectively positive-semidefinite and positive-definite for every $t \in [a,b]$. In particular, note that $\overline{\CC}(q,u) \geq 0$ for every $(q,u) \in \AC \times \L^2$.

Since $\langle S y , y \rangle_n = \langle \frac{1}{2} (S+S^\top) y , y \rangle_n$ for any $y \in \R^n$, we assume, without loss of generality, that $S$ is symmetric. For the same reason, we assume that $W(t)$ and $R(t)$ are symmetric for every $t \in [a,b]$.

\begin{remark}
In the case where the matrices $A$, $B$, $W$, $R$ and the functions $\omega$, $x$, $v$ are constant, Problem~\eqref{theproblem} is said to be \textit{autonomous}.
\end{remark}

\begin{remark}
In the case where $q_b = x(t) = \omega (t) = 0_{\R^n}$ and $v(t)= 0_{\R^m}$ for every $t \in [a,b]$, Problem~\eqref{theproblem} is said to be \textit{homogeneous}.
\end{remark}

The set $\E$ is used to model constraints on the set of \textit{sampling times} at which the value of the control can be modified. More precisely, we consider:
\begin{itemize}
\item[-] either \textit{permanent controls}, and then $\E = \L^2$ (no constraint). In this case, the value of the control $u$ can be modified at any time $t \in [a,b)$ and Problem~\eqref{theproblem} is said to be a general LQOCP with permanent controls.
\item[-] either \textit{sampled-data controls}, and then $\E$ is a set of piecewise constant controls with a fixed and finite number of switching times (see Section~\ref{section23}). In this case, we speak of \textit{nonpermanent controls} because the value of the control $u$ cannot be modified at any time, but only at fixed {sampling times}. More precisely we speak of \textit{sampled-data controls} because the value of the control $u$ is \textit{frozen} along each time interval between two consecutive sampling times (\textit{sample-and-hold procedure}). Problem~\eqref{theproblem} is said to be a general LQOCP with sampled-data controls.
\end{itemize}
The first objective of this paper is to prove that the optimal sampled-data controls converge pointwise to the optimal permanent control, when the distances between consecutive sampling times tend to zero (see Theorem~\ref{thmmainresult} in Section~\ref{section23}).

Before coming to that point, we first recall hereafter a series of well known results for Problem~\eqref{theproblem} (see, e.g., \cite{brys,lee,Schattler,trel}). Firstly the classical Cauchy-Lipschitz theorem leads to the two following results.

\begin{lemma}\label{propcauchyq}
For every $u \in \L^2$, there exists a unique solution $q \in \AC$ of the linear Cauchy problem given by
\begin{equation*}
\left\{\begin{split}
& \dot q(t) = A(t) q(t) + B(t) u(t)+\omega(t), \quad \text{a.e. $t \in [a,b]$,}  \\
& q(a)=q_a .
\end{split}\right.
\end{equation*}
We denote this solution by $q(\cdot,u)$.
\end{lemma}

\begin{lemma}\label{propcauchyp}
For every $u \in \L^2$, there exists a unique solution $p \in \AC$ of the backward linear Cauchy problem given by
\begin{equation*}
\left\{\begin{split}
& \dot p(t) = - A(t)^\top p(t) +W(t) \big( q(t,u)-x(t) \big), \quad \text{a.e. $t \in [a,b]$,}  \\
& p(b)= - S \big( q(b,u) -q_b \big).
\end{split}\right.
\end{equation*}
We denote this solution by $p(\cdot,u)$.
\end{lemma}

It clearly follows from Lemma~\ref{propcauchyq} that Problem~\eqref{theproblem} can be reduced to the minimization problem
$$ \min_{u \in \E} \; \CC(u), $$
where the cost functional $\CC : \L^2 \to \R^+$ is defined by $\CC(u) := \overline{\CC} ( q(\cdot ,u) ,u )$. For the reader's convenience, the proof of the following claim is recalled in Section~\ref{annexeproof}.

\begin{proposition}\label{thmexistenceunicitepermanent}
If $\E$ is a non-empty weakly closed convex subset of $\L^2$, then Problem~\eqref{theproblem} has a unique solution, denoted by $u^*_\E$.
\end{proposition}

In the permanent control case $\E = \L^2$, we denote the optimal solution by $u^* := u^*_{\L^2}$. The classical PMP (which is concerned with optimal permanent control problems) leads to the following necessary optimality condition.

\begin{proposition}\label{proppmp}
The optimal permanent control $u^*$ satisfies the implicit equality
\begin{equation}\label{eqvaleurdeupermanent}
u^* (t) = v(t)+ R(t)^{-1} B(t)^\top p(t,u^*),
\end{equation}
for almost every $t \in [a,b)$. Note that $u^*$ is (equal almost everywhere to) a continuous function on $[a,b]$. In particular, $u^*$ is bounded.
\end{proposition}

The proof of our first main result (stated in the next section) is based on the implicit equality~\eqref{eqvaleurdeupermanent}.

\subsection{Convergence result}\label{section23}
In order to define spaces of sampled-data controls, we first introduce the set
$$ \Delta := \bigcup_{N \in \N^*} \left\lbrace h= (h_i)_{i=0,\ldots,N-1} \in (0,+\infty)^N \text{ such that } \sum_{i=0}^{N-1} h_i = b-a  \right\rbrace .$$
For all $h \in \Delta$, we denote by $ \Vert h \Vert_\Delta := \max_{i=0,\ldots,N-1} h_i > 0$, and by $s^h_i := a + \sum_{j=0}^{i-1} h_j$ for all $i=0,\ldots,N$. In particular we have $a = s^h_0 < s^h_1 < \ldots  < s^h_N = b$.

\begin{center}
\begin{tikzpicture}[scale=1]
		\draw (0,0)--(10,0);
		\node at (0,0) {$\vert$}; \node at (2,0) {$\vert$}; \node at (5,0) {$\vert$}; \node at (8,0) {$\vert$}; \node at (10,0) {$\vert$};
		\node at (0,-0.5) {$a = s^h_0$}; \node at (2,-0.5) {$s^h_1$}; \node at (5,-0.5) {$s^h_2$}; \node at (8,-0.5) {$s^h_{N-1}$}; \node at (10,-0.5) {$s^h_{N} = b$};
		\node at (6.5,-0.5) {$\ldots$};
		\node at (1,0.5) {$h_0$}; \node at (3.5,0.5) {$h_1$}; \node at (9,0.5) {$h_{N-1}$};
\end{tikzpicture}
\end{center}

Finally, for all $h \in \Delta$, we introduce the space $\E_h$ of sampled-data controls 
$$ \E_h := \{ u : [a,b) \rightarrow \R^m \; \mid \; u_{| [s^h_i,s^h_{i+1})} \text{ is constant for every $i=0,\ldots,N-1$} \} \subset \L^2. $$
Note that $ u : [a,b) \rightarrow \R^m$ belongs to $\E_h$ if and only if $u = \sum_{i=0}^{N-1} U_{i} \; \mathbf{1}_{[s^h_i,s^h_{i+1})}$, for some $U_i \in \R^m$ for every $i=0,\ldots,N-1$, where $\mathbf{1}$ denotes the indicator function.

\begin{remark}
Considering Problem $(\mathscr{P}_{\E_h})$, note that $(s^h_i)_{i=0,\ldots,N-1}$ play the role of fixed {sampling times} at which (and only at which) the value of the control $u$ can be modified.
\end{remark}

It is clear that $\E_h$ is a non-empty weakly closed convex subset of $\L^2$ for all $h \in \Delta$. From Proposition~\ref{thmexistenceunicitepermanent}, Problem~($\mathscr{P}_{\E_h}$) admits a unique solution denoted by $u^*_h := u^*_{\E_h} \in \E_h$, that is, $u^*_h$ is the optimal sampled-data control associated to $h \in \Delta$.

The PMP recently stated in \cite{bourdin-trelat-pontryagin2} (which can be applied to optimal sampled-data control problems) leads to the following necessary optimality condition.

\begin{proposition}\label{proppmpsample}
Let $h \in \Delta$. The optimal sampled-data control $u^*_h$, written as
\begin{equation}\label{expression_u_hstar}
u^*_h = \sum_{i=0}^{N-1} U^*_{h,i} \; \mathbf{1}_{[s^h_i,s^h_{i+1})},
\end{equation}
with $ U^*_{h,i} \in \R^m$ for every $i=0,\ldots,N-1$,
satisfies the implicit equality
\begin{equation}\label{eqvaleurdeusample}
U^*_{h,i} = \left( \dfrac{1}{h_i} \di \int_{s^h_i}^{s^h_{i+1}} R(s) \; ds \right)^{-1} \left( \dfrac{1}{h_i} \di \int_{s^h_i}^{s^h_{i+1}} R(s) v(s)  + B(s)^\top p(s,u^*_h) \; ds \right),
\end{equation}
for every $i=0,\ldots,N-1$.
\end{proposition}

The first main result of this paper is the following.

\begin{theorem}\label{thmmainresult}
The sequence $(u^*_h)_{h \in \Delta}$ of optimal sampled-data controls converges pointwise on $[a,b)$ to the optimal permanent control $u^*$ as $\Vert h \Vert_\Delta$ tends to $0$.
\end{theorem}

\begin{proof}
Theorem~\ref{thmmainresult} follows from the apparent relationship between the implicit equalities~\eqref{eqvaleurdeupermanent} and \eqref{eqvaleurdeusample}, and from the continuity of $R$, $B$, $v$ and $p$. Actually we only need to prove that $p(\cdot,u^*_h) \to p(\cdot,u^*)$ in $\C$. To this end, we introduce for all $h \in \Delta$ the sampled-data control $u_h \in \E_h$ defined by $ u_h := \sum_{i=0}^{N-1} U_{h,i} \; \mathbf{1}_{[s^h_i,s^h_{i+1})}$, where
$$ U_{h,i} := \di \frac{1}{h_i} \int_{s^h_i}^{s^h_{i+1}} u^* (s) \; ds ,$$
for every $i=0,\ldots,N-1$. Since $u^*$ is continuous on $[a,b]$, it is clear that $u_h (t)$ converges to $u^*(t)$ for every $t \in [a,b]$, and from the classical Lebesgue dominated convergence theorem that $u_h \to u^*$ in $\L^2$, when $\Vert h \Vert_\Delta$ tends to $0$. From Lemma~\ref{groslem} in Section~\ref{annexeproof}, we conclude that $\CC (u_h) $ tends to $\CC(u^*)$. By optimality of $u^*$ and $u^*_h$, we have $ \CC (u^*) \leq \CC (u^*_h) \leq \CC (u_h) $ for all $h \in \Delta$, and we get that $\CC (u^*_h) $ tends to $\CC(u^*)$ when $\Vert h \Vert_\Delta$ tends to $0$. Since $\CC (u^*_h) $ tends to $\CC(u^*)$, we conclude that $(u^*_h)_{h \in \Delta}$ is a minimizing sequence of $\CC$ on $\L^2$ and, using the same arguments as in the proof of Proposition~\ref{thmexistenceunicitepermanent} (see Section~\ref{annexeproof}), we deduce that (up to a subsequence that we do not relabel) $u^*_h \rightharpoonup u^*$ in $\L^2$. Actually one can easily prove by contradiction that the whole sequence $(u^*_h)_{h \in \Delta}$ weakly converges to $u^*$ in $\L^2$. Finally Lemma~\ref{groslem} in Section~\ref{annexeproof} concludes the proof.
\end{proof}

\begin{remark}\label{rem}
It follows from the above proof and from Lemma~\ref{groslem} in Section~\ref{annexeproof} that 
$$
\CC(u^*_h)\rightarrow \CC(u^*),\quad u^*_h \rightharpoonup u^*\ \textrm{in}\ L^2,\quad
q(\cdot,u^*_h) \rightarrow q(\cdot,u^*)\ \textrm{in}\ \C,
$$
as $\Vert h \Vert_\Delta\rightarrow 0$.
\end{remark}

\begin{remark}\label{assumptionsweakened}
The results of Theorem \ref{thmmainresult} and Remark~\ref{rem} remain valid under the following weakened regularity assumptions on the data of Problem~\eqref{theproblem}:
\begin{itemize}
\item[-] $A$ and $W$ are integrable matrices, $B$ is a $\ell$-integrable matrix with $\ell > 2$, and $R$ is an essentially bounded matrix;
\item[-] $x$ and $\omega$ are integrable functions, and $v$ is a square-integrable function;
\item[-] The products $Wx$ and $W^\top x$ and the scalar product $\langle Wx,x \rangle$ are integrable functions;
\item[-] $W(t)$ is positive-semidefinite for almost every $t \in [a,b]$;
\item[-] There exists a constant $c_R > 0$ such that $\langle R(t) z,z \rangle_m \geq  c_R \Vert z \Vert^2_m$ for almost every $t \in [a,b]$ and every $z \in \R^m$.
\end{itemize}
One can easily adapt the proof of Theorem~\ref{thmmainresult} (by considering Lebesgue points) and prove that the  convergence of the sequence $(u^*_h)_{h \in \Delta}$ to $u^*$ is still valid, but only almost everywhere on $[a,b)$. 
\end{remark}

\begin{remark}\label{controlaveraged}
In the above proof, $u_h$ is introduced as the sampled-data control whose values correspond to the averages of the optimal permanent control $u^*$ on each sampling interval $[s_i^h,s_{i+1}^h)$. Let us mention that $u_h \neq u^*_h$ in general (see Figure~\ref{figure2} in Section~\ref{section4} for an example). Choosing $\varphi_h = u_h$ as a sampled-data control does not lead in general to an optimal rate of convergence of $\CC (\varphi_h) $ to $\CC(u^*)$ when $\Vert h \Vert_\Delta$ tends to $0$. On the other hand, from the optimality of $u^*_h$, it is clear that choosing $\varphi_h = u^*_h$ does. Moreover, choosing $\varphi_h = u_h$ requires the knowledge of $u^*$, while choosing $\varphi_h = u^*_h$ does not. Indeed, we provide in the next section a recursive way allowing to compute explicitly the optimal coefficients $U^*_{h,i}$ for every $i=0,\ldots,N-1$ (see Corollary~\ref{corollary} in Section~\ref{secmainresult2}).
\end{remark}

\section{Riccati theory for optimal sampled-data controls}\label{section3}
In this section, we fix $h \in \Delta$ and our objective is to provide an expression for the optimal sampled-data control $u^*_h$ as a closed-loop control (see Theorem~\ref{thmmainresult2} in Section~\ref{secmainresult2}). This corresponds to an extension of the classical Riccati theory to the sampled-data control case. Moreover, we will show that our extension of the Riccati theory allows to compute explicitly (and in a recursive way) the optimal coefficients $U^*_{h,i}$ for every $i=0,\ldots,N-1$ (see Corollary~\ref{corollary} in Section~\ref{secmainresult2}).

To be in accordance with the classical literature on Riccati theory, and for the sake of completeness, we will provide two different proofs of Theorem~\ref{thmmainresult2}. The first proof (see Section~\ref{secfirstproof}) is based on Proposition~\ref{proppmpsample}, \textit{i.e.}, on the PMP recently stated in \cite{bourdin-trelat-pontryagin2}. The second proof (see Section~\ref{secsecondproof}) is based on the dynamic programming principle, and extends a strategy used in \cite{bini}.

For the ease of notations, since $h \in \Delta$ is fixed throughout Section~\ref{section3}, we set $s_i := s^h_i$ for all $i=0,\ldots,N$.

\subsection{Some notations}\label{secmatrice}
For every $s \in [a,b]$, we denote by $Z(\cdot,s) : [a,b] \rightarrow \R^{n,n}$ the unique solution of the backward/forward linear Cauchy problem given by
\begin{equation*}
\left\{\begin{split}
& \dot Z(t) = A(t) Z(t), \quad \text{for every $t \in [a,b]$,}  \\
& Z(s)= \mathrm{Id}_n ,
\end{split}\right.
\end{equation*}
and $Z(\cdot,\cdot)$ is the so-called \textit{state-transition matrix} associated to $A$.

Let us introduce the following terms that will play an important role in the sequel:
$$ \ZB_i := \di \int_{s_i}^{s_{i+1}} Z(s_{i+1}, s ) B(s ) \; ds  \in \R^{n,m} , $$
$$ \ZO_i := \left\lbrace \begin{array}{lcl}
\di \int_{s_i}^{s_{i+1}} Z(s_{i+1},s ) \omega(s ) \; ds  \in \R^{n}, & \text{if} & i = 1, \ldots ,N-2, \\[15pt]
\di \int_{s_i}^{s_{i+1}} Z(s_{i+1},s ) \omega(s ) \; ds  - q_b \in \R^{n}, & \text{if} & i = N-1,
\end{array}
 \right. $$
$$ \ZWZ_i := \di \int_{s_i}^{s_{i+1}} Z(\t,s_i)^\top W(\t) Z(\t,s_i) \; d\t \in \R^{n,n}, $$
$$ \ZBWZ_i := \di \int_{s_i}^{s_{i+1}} \left( \int_{s_i}^{\t} Z(\t,s)B(s) \; ds \right)^\top {W}(\t) Z(\t,s_i) \; d\t  \in \R^{m,n}, $$
$$ \ZBWZB_i := \di \int_{s_i}^{s_{i+1}} \left( \int_{s_i}^{\t} Z(\t,s)B(s) \; ds \right)^\top W(\t)  \left( \int_{s_i}^{\t} Z(\t,s)B(s) \; ds \right) \; d\t  \in \R^{m,m} , $$
$$ \ZBWZOX_i := \di \int_{s_i}^{s_{i+1}} \left( \int_{s_i}^{\t} Z(\t,s)B(s) \; ds \right)^\top {W}(\t) \left( \int_{s_i}^{\t} Z(\t,s) \omega (s) \; ds - x(\t) \right) \; d\t  \in \R^{m}, $$
$$ \ZWZOX_i := \di \int_{s_i}^{s_{i+1}} Z(\t,s_i)^\top {W}(\t) \left( \int_{s_i}^{\t} Z(\t,s) \omega (s) \; ds - x(\t) \right) \; d\t \in \R^{n},  $$
$$ \WZOX^2_i :=  \di \int_{s_i}^{s_{i+1}} \left\langle W(\t) \left( \int_{s_i}^{\t} Z(\t,s) \omega (s) \; ds -x(\t)  \right) , \left( \int_{s_i}^{\t} Z(\t,s) \omega (s) \; ds -x(\t)  \right) \right\rangle \; d\t \in \R,$$
$$ \RR_i := \di \int_{s_i}^{s_{i+1}} R(s) \; ds   \in \R^{m,m},$$
$$ \RV_i := \di \int_{s_i}^{s_{i+1}} R(s) v(s) \; ds   \in \R^{m},  $$
$$ \RV^2_i := \di \int_{s_i}^{s_{i+1}} \big\langle  R(\t) v(\t) , v(\t) \big\rangle_m \; d\t   \in \R,  $$
for every $i=0,\ldots,N-1$.

Finally we introduce $(\KK_i)_{i=0,\ldots,N}$, $(\JJ_i)_{i=0,\ldots,N}$ and $(\YY_i)_{i=0,\ldots,N}$ the following backward recursive sequences:
\begin{equation*}
\left\{\begin{split}
& \KK_{N}=S , \\
& \KK_i = \QQ_i - \PP_i^\top \TT_i^{-1} \PP_i \in \R^{n,n}, \quad \text{for every $i=N-1,\ldots,0$,}
\end{split}\right.
\end{equation*}
and
\begin{equation*}
\left\{\begin{split}
& \JJ_{N}=0_{\R^n} , \\
& \JJ_i = \GG_i - \PP_i^\top \TT_i^{-1} \HH_i \in \R^n, \quad \text{for every $i=N-1,\ldots,0$,}
\end{split}\right.
\end{equation*}
and
\begin{equation*}
\left\{\begin{split}
& \YY_{N}=0 , \\
& \YY_i = \FF_i - \langle \TT_i^{-1} \HH_i, \HH_i \rangle_m \in \R, \quad \text{for every $i=N-1,\ldots,0$,}
\end{split}\right.
\end{equation*}
where $\FF_i$, $\GG_i$, $\HH_i$, $\PP_i$, $\QQ_i$ and $\TT_i$ are defined (explicitly and dependently on $\KK_{i+1}$, $\JJ_{i+1}$ and $\YY_{i+1}$) as follows:
$$ \FF_i :=  \langle \KK_{i+1} \ZO_i ,  \ZO_i \rangle_n + \WZOX_i^2 + \RV_i^2 + 2 \langle \JJ_{i+1} , \ZO_i \rangle_n + \YY_{i+1}  \in \R, $$
$$ \GG_i := Z(s_{i+1},s_i)^\top {\KK}_{i+1} \ZO_i + \ZWZOX_i + Z(s_{i+1},s_i)^\top \JJ_{i+1} \in \R^{n},  $$
$$ \HH_i := \ZB^\top_i {\KK}_{i+1} \ZO_i + \ZBWZOX_i -  \RV_i + \ZB_i^\top \JJ_{i+1}  \in \R^{m}, $$
$$ \PP_i := \ZB_i^\top {\KK}_{i+1} Z(s_{i+1},s_i) + \ZBWZ_i \in \R^{m,n}, $$
$$ \QQ_i := Z(s_{i+1},s_i)^\top \KK_{i+1} Z(s_{i+1},s_i) + \ZWZ_i \in \R^{n,n}, $$
$$ \TT_i := \ZB_i^\top \KK_{i+1} \ZB_i + \ZBWZB_i + \RR_i  \in \R^{m,m}, $$
for every $i=N-1,\ldots,0$.

\begin{remark}\label{remTiinv}
A necessary condition for the backward sequences $(\KK_i)_{i=0,\ldots,N}$, $(\JJ_i)_{i=0,\ldots,N}$ and $(\YY_i)_{i=0,\ldots,N}$ to be well defined, is the invertibility of ${\TT}_i$ for every $i=N-1,\ldots,0$. This necessary condition will be established in Section~\ref{secfirstproof} (see also Section~\ref{secsecondproof}). More precisely, we will prove in a backward recursive way that $\KK_{i+1}$ is positive-semidefinite for every $i=N-1,\ldots,0$. As a consequence, $\TT_i$ is equal to a sum of two positive-semidefinite matrices $\ZB_i^\top \KK_{i+1} \ZB_i$ and $\ZBWZB_i$ (see Remark~\ref{rempositivematrix2}) and of a positive-definite matrix $\RR_i$. We deduce that $\TT_i$ is positive-definite and hence it is invertible for every $i=N-1,\ldots,0$.
\end{remark}

\begin{remark}\label{termdependant}
All terms introduced in this section depend only on the data of Problem~$(\mathscr{P}_{\E_h})$, \textit{i.e.}, on $A$, $B$, $S$, $W$, $R$, $\omega$, $q_b$, $x$, $v$ and $h$. It is worth to note that they do not depend on the initial condition~$q_a$. As a consequence, all these terms, that are defined in a backward recursive way, remain unchanged if the initial condition in Problem~$(\mathscr{P}_{\E_h})$ is modified, and they remain unchanged as well if the initial time $a$ is replaced by $s_j$ for some $j=0,\ldots, N-1$ and $h$ is replaced by $(h_i)_{i=j,\ldots,N-1}$.
\end{remark}

\begin{remark}\label{cashomogene}
In the homogeneous case, we have $\ZO_i = \ZWZOX_i = \GG_i = \JJ_i = 0_{\R^n}$, $\ZBWZOX_i = \RV_i = \HH_i = 0_{\R^m}$, $\WZOX^2_i = \RV^2_i = \FF_i =\YY_i = 0$ for every $i=0,\ldots,N$, and then many simplifications occur in the formulas.
\end{remark}

\begin{remark}\label{Kindependant}
The sequences $(\KK_i)_{i=0,\ldots,N}$, $(\PP_i)_{i=0,\ldots,N}$, $(\QQ_i)_{i=0,\ldots,N}$ and $(\TT_i)_{i=0,\ldots,N}$ do not depend on the nonhomogeneous data $q_b$, $\omega$, $x$ and $v$. As a consequence, they remain unchanged regardless of whether we consider the homogeneous or the nonhomogeneous Problem~$(\mathscr{P}_{\E_h})$.
\end{remark}

\subsection{Closed-loop optimal control}\label{secmainresult2}
The second main result of this paper is the following.

\begin{theorem}\label{thmmainresult2}
The optimal sampled-data control $u^*_h$, defined by \eqref{expression_u_hstar}, is given in a closed-loop form as
$$ U^*_{h,i} = - {\TT}_i^{-1} \big( \PP_i q (s_i , u^*_h) + \HH_i \big),$$
for every $i=0,\ldots,N-1$. Moreover, the corresponding optimal cost is
$$ \CC(u^*_h) = \dfrac{1}{2} \langle \KK_0 q_a , q_a \rangle_n + \langle \JJ_0, q_a \rangle_n +\dfrac{1}{2} \YY_0 .$$
\end{theorem}

Two different proofs of Theorem~\ref{thmmainresult2} are given in Sections~\ref{secfirstproof} and \ref{secsecondproof}.

\begin{corollary}\label{corollary}
The optimal values $(U^*_{h,i})_{i=0,\ldots,N-1}$ can be explicitly computed by induction,
$$ \left\lbrace \begin{array}{l}
U^*_{h,i} = - {\TT}_i^{-1} \big( \PP_i q_i + \HH_i \big), \\[5pt]
q_{i+1} = Z(s_{i+1},s_{i}) q_i + \ZB_{i} U^*_{h,i} + \ZO_i,
\end{array} \right. $$
for every $i=0,\ldots,N-1$, with the initial condition $q_0 = q_a$.
\end{corollary}

\begin{proof}
The Duhamel formula gives $ q (s_{i+1} , u^*_h) = Z(s_{i+1},s_{i}) q (s_i , u^*_h) + \ZB_{i} U^*_{h,i} + \ZO_i$, for every $i=0,\ldots,N-1$. Corollary~\ref{corollary} follows using Theorem~\ref{thmmainresult2}.
\end{proof}

As a conclusion, in order to compute explicitly the optimal coefficients $(U^*_{h,i})_{i=0,\ldots,N-1}$, one has beforehand to compute all terms introduced in Section~\ref{secmatrice} (they depend only on the data of Problem~($\mathscr{P}_{\E_h}$), see Remark~\ref{termdependant}). Secondly, one has to compute the induction provided in Corollary~\ref{corollary}.

\subsection{Some numerical simulations for a simple example}\label{section4}

In this section we focus on the unidimensional LQOCP given by
\begin{equation}\label{theproblemsection4}\tag{$\QQQ$}
\begin{array}{ll}
\text{minimize} & \di \int_0^1  q(\t)^2 + \frac{1}{2} u(\t)^2  \; d\t, \\
& \\
\text{subject to} & q \in \AC, \quad u \in \E \subset \L^2, \\[5pt]
& \dot q(t) = \frac{1}{2} q(t) + u(t), \quad \text{a.e. $t \in [0,1]$,}  \\[3pt]
& q(0)=1.
\end{array}
\end{equation}
It is clear that the data of Problem~\eqref{theproblemsection4} satisfy all assumptions of Section~\ref{secprelim}. This very simple problem has been considered in \cite{dontchev,hager,hager2}, where the authors were interested in convergence issues for specific discretizations, showing that the simplest direct method diverges when considering an explicit second-order Runge-Kutta discretization. 
This is why this apparently inoffensive example is interesting and this is why we consider it here as well.

In the permanent control case $\E = \L^2$, the unique optimal permanent control $u^*$ is  given by
$$ \forall t \in [0,1], \quad u^* (t) = \dfrac{2(e^{3t}-e^3)}{e^{3t/2} (2+e^3)}. $$
In this section, we are interested in the unique solution $u^*_h$ of Problem~$(\mathscr{Q}_{\E_h})$ for different values of $h \in \Delta$. More precisely, we take $h = \frac{1}{N} (1,\ldots,1) \in (0,+\infty)^N$ with different values of $N \in \N^*$. Computing the induction provided in Corollary~\ref{corollary}, we obtain the numerical results depicted in Figure~\ref{figure1}. When $\Vert h \Vert_\Delta$ tends to $0$, we observe as expected (see Theorem~\ref{thmmainresult}) the pointwise convergence of $u^*_h$ to $u^*$.
\begin{figure}[h]
\centering
\includegraphics[scale=0.35]{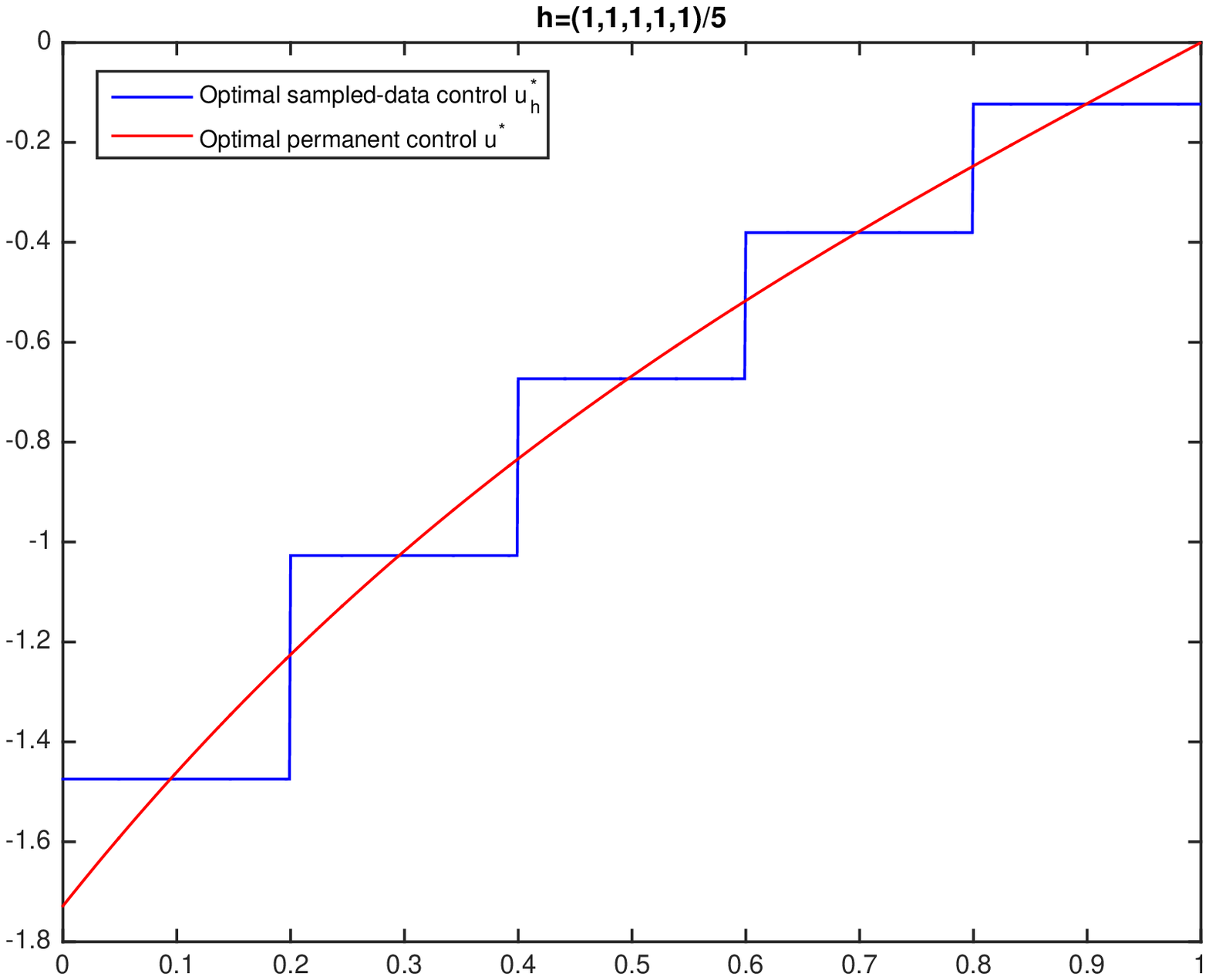}  \includegraphics[scale=0.35]{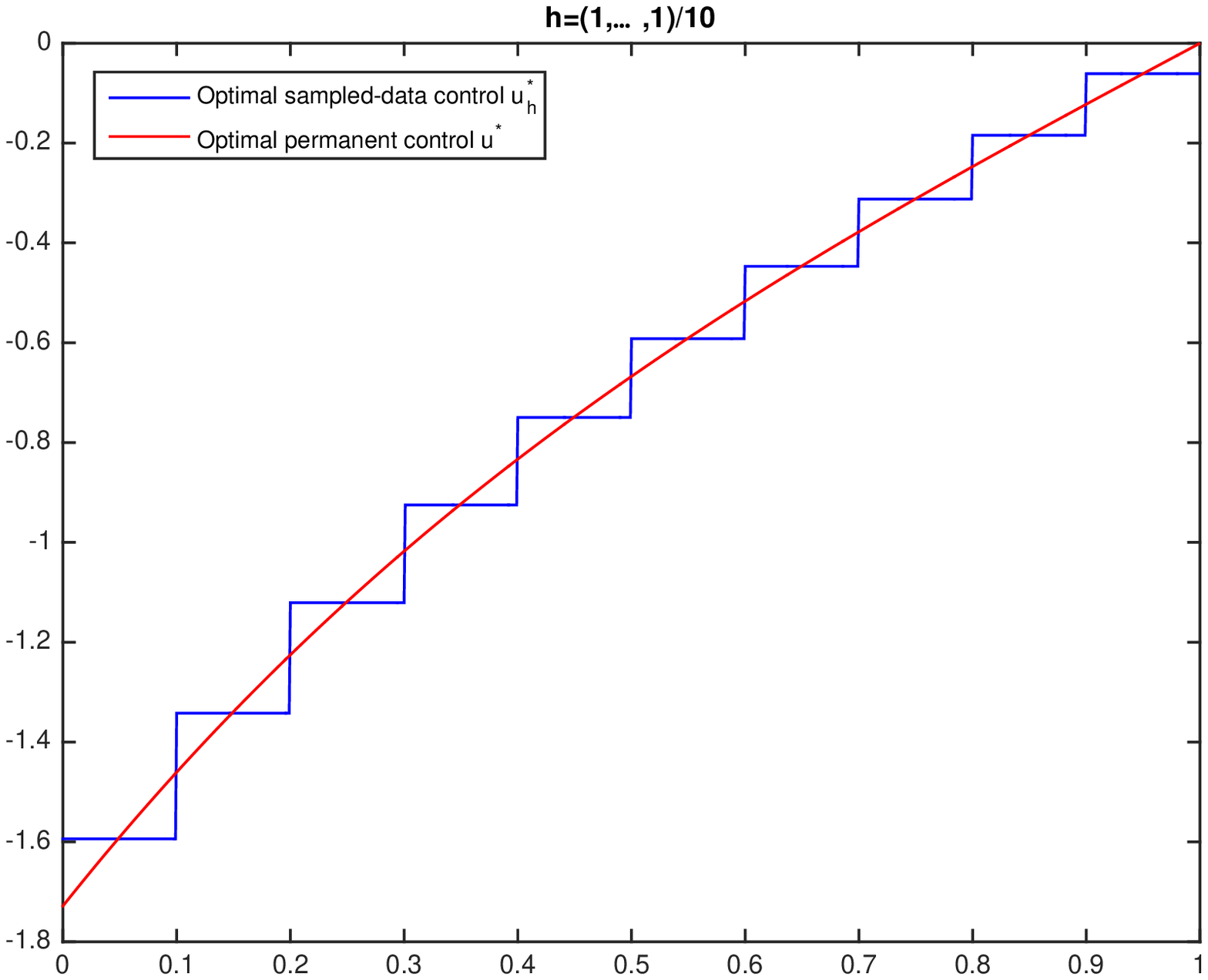} \\
\includegraphics[scale=0.35]{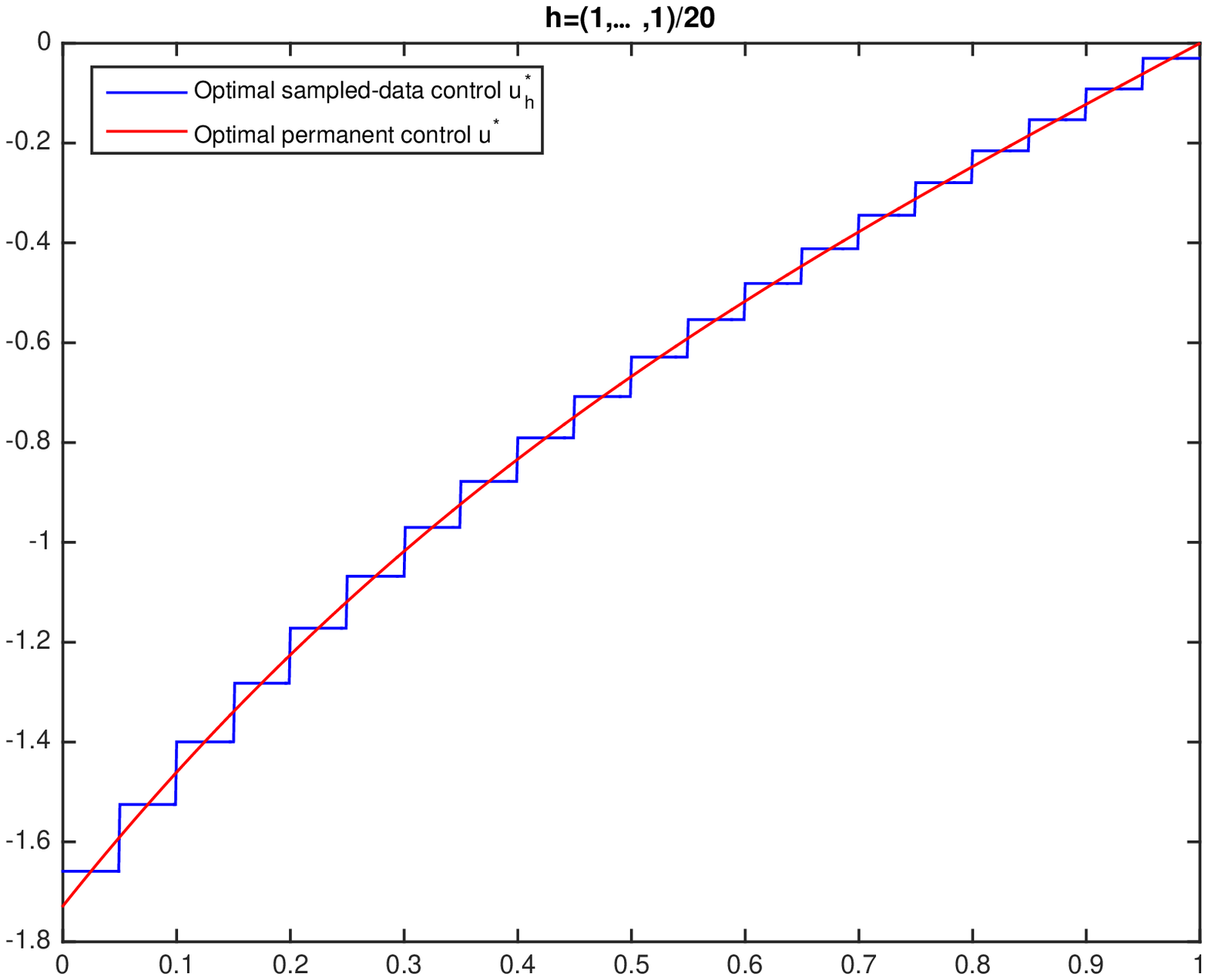}  \includegraphics[scale=0.35]{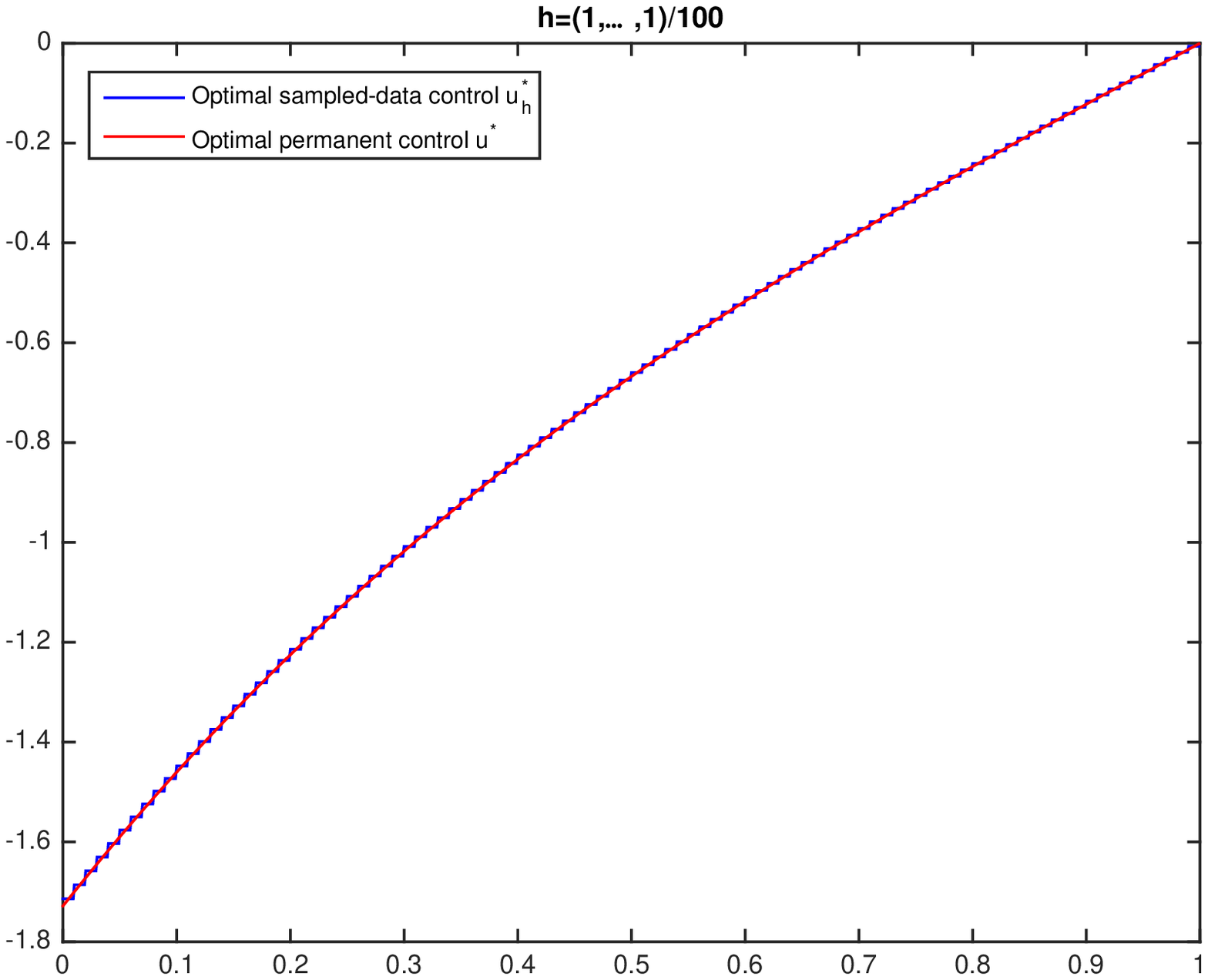}
\caption{Pointwise convergence of $u^*_h$ to $u^*$ as $\Vert h \Vert_\Delta$ tends to $0$.}
\label{figure1}
\end{figure}

Figure~\ref{figure2} represents the sampled-data control $u_h$ introduced in the proof of Theorem~\ref{thmmainresult}. We observe as expected (see Remark~\ref{controlaveraged}) that $u_h \neq u^*_h$.

\begin{figure}[h]
\centering
\includegraphics[scale=0.35]{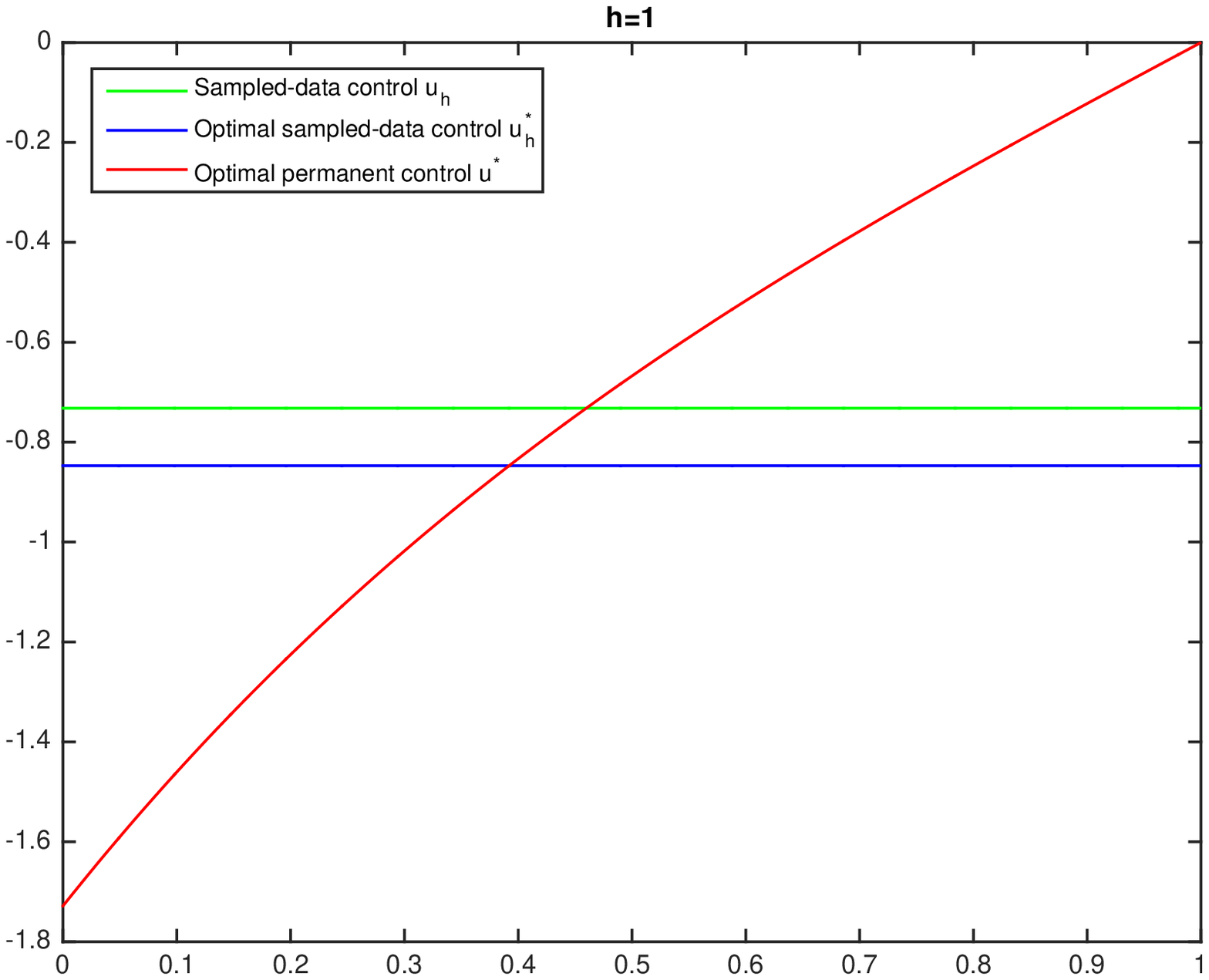}  \includegraphics[scale=0.35]{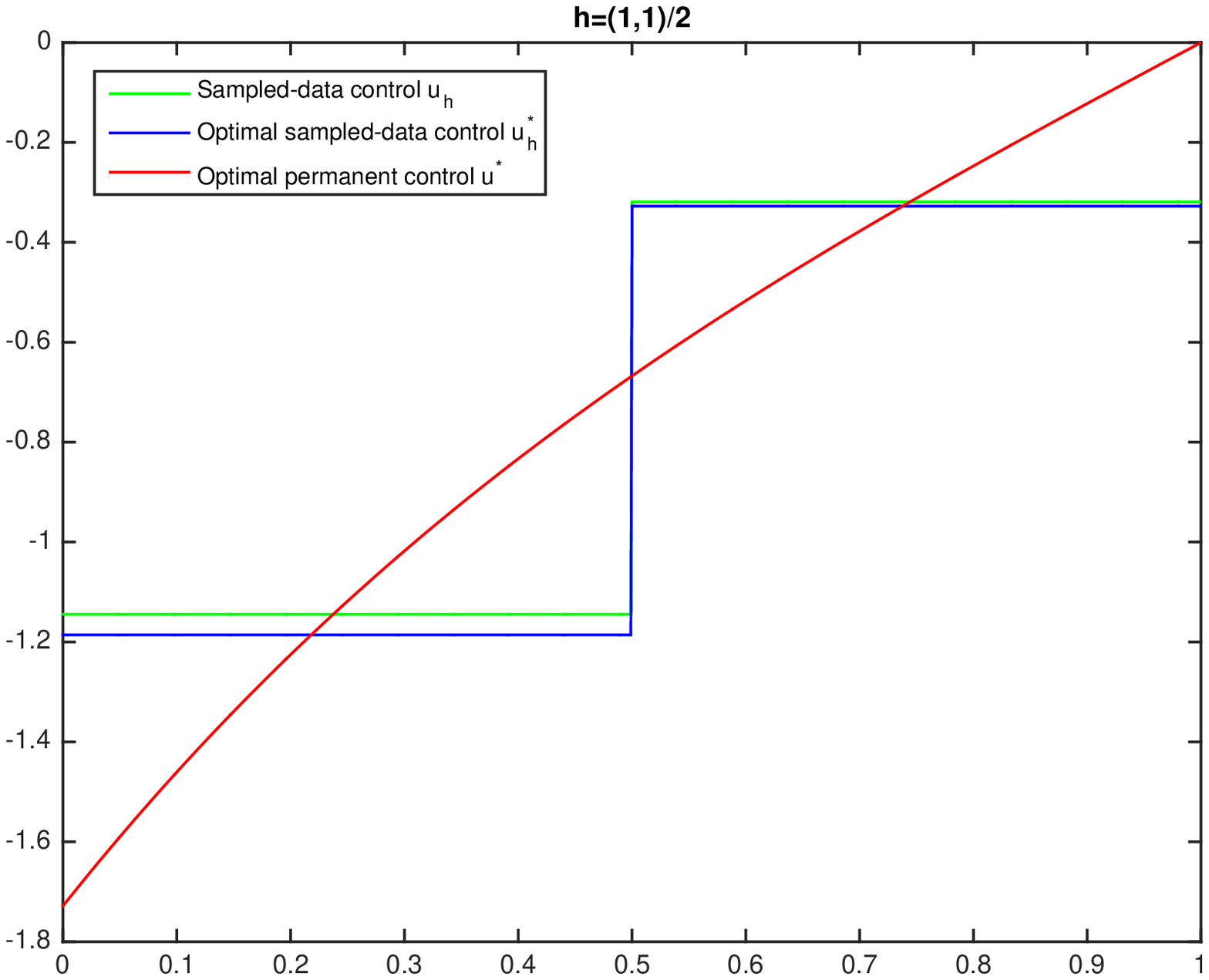}
\caption{In general $u_h \neq u^*_h$.}
\label{figure2}
\end{figure}

\section{Proofs}

\subsection{Proof of Proposition~\ref{thmexistenceunicitepermanent}}\label{annexeproof}

In order to prove Proposition~\ref{thmexistenceunicitepermanent}, we first state and prove two preliminary lemmas, variants of which are well known in the existing literature (see, e.g., \cite{brys,lee,Schattler,trel}). The proofs are given for the sake of completeness.

\begin{lemma}\label{lemequiHolder}
Let $(g_k)_{k \in \N}$ be a sequence of functions defined on $[a,b]$ with values in $\R^n$. We assume that $(g_k)_{k \in \N}$ is equi-H\"olderian in the sense that there exist $\alpha > 0$ and $\beta>0$ such that $ \Vert g_k (t_2) - g_k (t_1) \Vert_n \leq \beta \vert t_2 - t_1 \vert^\alpha $
for every $k \in \N$ and every $t_1$, $t_2 \in [a,b]$. If the sequence $ (g_k)_{k \in \N}$ converges pointwise on $[a,b]$ to $0$, then it converges uniformly on $[a,b]$ to $0$.
\end{lemma}

\begin{proof}
Let $\varepsilon > 0$ and let $a = t_0 < \ldots < t_p = b$ be a partition of $[a,b]$ such that $t_{i+1} - t_i < (\frac{\varepsilon}{2\beta})^{1/\alpha}$ for every $i=0,\ldots,p-1$. From the pointwise convergence, there exists $K \in \N$ such that $\Vert g_k (t_i) \Vert_n < \frac{\varepsilon}{2}$ for every $k \geq K$ and every $i=0,\ldots,p$. Let $k \geq K$. Then, for every $t \in [a,b]$, $t \in [t_i,t_{i+1}]$ for some $i=0,\ldots,p-1$ and it follows that $ \Vert g_k (t) \Vert_n \leq \Vert g_k (t) - g_k (t_i) \Vert_n + \Vert g_k (t_i) \Vert_n < \varepsilon$. The proof is complete.
\end{proof}

\begin{lemma}\label{groslem}
The following properties hold true:
\begin{enumerate}
\item If $u_k \rightharpoonup u$ in $\L^2$, then $q(\cdot,u_k) \rightarrow q(\cdot,u)$ in $\C$ and $p(\cdot,u_k) \rightarrow p(\cdot,u)$ in $\C$.
\item The cost functional $\CC$ is strictly convex on $\L^2$.
\item If $ u_k \rightharpoonup u$ in $\L^2$, then $\liminf_{k \to \infty} \CC(u_k) \geq \CC(u)$.
\item If $ u_k \to u$ in $\L^2$, then $\lim_{k \to \infty} \CC(u_k) = \CC(u)$.
\end{enumerate}
\end{lemma}

\begin{proof}
\textbf{\textit{1.}} Let us assume that $ u_k \rightharpoonup u$ in $\L^2$. For every $k \in \N$ and every $t \in [a,b]$, let us define $ g_k (t) := \int_a^t B(\t) (u_k(\t)-u(\t))  d\t$. Since $ u_k \rightharpoonup u$ in $\L^2$, the sequence $ (g_k)_{k \in \N}$ converges pointwise on $[a,b]$ to $0$. Moreover, for every $k \in \N$ and every $t_1$, $t_2 \in [a,b]$, it follows from the classical H\"older inequality that $ \Vert g_k(t_2) - g_k(t_1) \Vert_n \leq \vertiii{B}_{\infty} \Vert u_k - u \Vert_{\L^2} \vert t_2 - t_1 \vert^{1/2} $. Since $ u_k \rightharpoonup u$ in $\L^2$, the term $\Vert u_k - u \Vert_{\L^2}$ is bounded and it follows that the sequence $(g_k)_{k \in \N}$ is equi-H\"olderian. From Lemma~\ref{lemequiHolder}, the sequence $ (g_k)_{k \in \N}$ converges uniformly on $[a,b]$ to $0$. Finally, the classical Gronwall lemma leads to $ \Vert q(t,u_k) - q(t,u) \Vert_n \leq \Vert g_k (t) \Vert_n + \vertiii{A}_{\infty} \int_a^t \Vert q(\t,u_k) - q(\t,u) \Vert_n \; d\t \leq \Vert g_k \Vert_{\infty} e^{\vertiii{A}_{\infty} (b-a)}$ for every $t \in [a,b]$ and every $k \in \N$. We deduce that $q(\cdot,u_k) \rightarrow q(\cdot,u)$ in $\C$. One can similarly derive that $p(\cdot,u_k) \rightarrow p(\cdot,u)$ in $\C$.

\textbf{\textit{2.}} Since $S$ is positive-semidefinite, $W(t)$ is positive-semidefinite and $R(t)$ is positive-definite for every $t \in [a,b]$, the functional $\overline{\CC}$ is clearly convex in the variable $q$ and strictly convex in the variable $u$. Moreover we have $ q(\cdot,\lambda u_1 + (1-\lambda) u_2) = \lambda q(\cdot,u_1) + (1-\lambda) q(\cdot ,u_2)$ for every $u_1$, $u_2 \in \L^2$ and every $\lambda \in [0,1]$. One can easily deduce the strict convexity of the cost functional $\CC$ from these facts.

\textbf{\textit{3.}}  From the hypotheses on $R$, one can easily prove by contradiction that there exists a constant $c_R > 0$ such that $ c_R \Vert z \Vert^2_m \leq \langle R(t)z,z \rangle_m \leq \vertiii{R}_{\infty} \Vert z \Vert^2_m$ for every $t \in [a,b]$ and every $z \in \R^m$. As a consequence, the scalar product $\langle \cdot , \cdot \rangle_R$ defined on $\L^2$ by
$$ \langle u_1 , u_2 \rangle_R := \int_a^b \big\langle R(\t)u_1(\t),u_2 (\t) \big\rangle_m \; d\t $$
induces a norm $\Vert \cdot \Vert_R$ on $\L^2$ that is equivalent to the usual one. Let us assume that $ u_k \rightharpoonup u$ in $\L^2$. Since $q(\cdot,u_k) \rightarrow q(\cdot,u)$ in $\C$, we have
\begin{multline*}
 \lim\limits_{k \to \infty} \dfrac{1}{2} \Big\langle S \big( q(b,u_k)-q_b \big) , q(b,u_k)-q_b \Big\rangle_n + \dfrac{1}{2} \di \int_a^b \Big\langle W(\t) \big( q(\t,u_k)-x(\t) \big) , q(\t,u_k) -x(\t) \Big\rangle_n \; d\t  \\ = \dfrac{1}{2} \Big\langle S \big( q(b,u)-q_b \big) , q(b,u)-q_b \Big\rangle_n + \dfrac{1}{2} \di \int_a^b \Big\langle W(\t) \big( q(\t,u)-x(\t) \big) , q(\t,u) -x(\t) \Big\rangle_n \; d\t.
\end{multline*}
Moreover, since $ u_k \rightharpoonup u$ in $\L^2$, we have $ \liminf_{k \to \infty} \frac{1}{2} \Vert u_k - v \Vert^2_R \geq \frac{1}{2} \Vert u - v \Vert^2_R $. This concludes the proof.

\textbf{\textit{4.}} The proof is  similar since $ u_k \to u$ in $\L^2$ implies that $\lim_{k \to \infty} \frac{1}{2} \Vert u_k -v\Vert^2_R = \frac{1}{2} \Vert u-v \Vert^2_R$.
\end{proof}

We are now in a position to prove Proposition~\ref{thmexistenceunicitepermanent}. Let us prove that $\CC$ has a unique minimizer on $\E$. Uniqueness is clear since $\E$ is convex and $\CC$ is strictly convex (see Lemma~\ref{groslem}). Now let us prove existence. Let $(u_k)_{k \in \N} \subset \E$ be a minimizing sequence of $\CC$ on $\E$. Since $\CC(u_k) \geq \frac{1}{2} \Vert u_k-v \Vert^2_R$, we conclude that $(u_k)_{k \in \N}$ is bounded in $\L^2$ and thus converges weakly, up to some subsequence, to some $u^*_\E$. Since $\E$ is weakly closed, we get that $u^*_\E \in \E$. Finally, from Lemma~\ref{groslem}, we get that $ \inf_{u \in \E}  \CC(u) =  \lim_{k \to \infty} \CC(u_k) = \liminf_{k \to \infty} \CC(u_k) \geq \CC(u^*_\E)$ which concludes the proof.

\subsection{Preliminaries for Theorem~\ref{thmmainresult2} and value function}\label{secVF}
In this section, we establish preliminary results that are required to prove Theorem~\ref{thmmainresult2}. Precisely they are required in order to prove the invertibility of the matrices $\TT_i$ in the first proof of Theorem~\ref{thmmainresult2} (based on Proposition~\ref{proppmpsample} and detailed in Section~\ref{secfirstproof}), and to prove the dynamic programming principle which is the basis of the second proof of Theorem~\ref{thmmainresult2} (detailed in Section~\ref{secsecondproof}).

The reader who would be interested in the proof of Theorem~\ref{thmmainresult2} based on Proposition~\ref{proppmpsample} and who wants to skip technical difficulties related to the invertibility of the matrices $\TT_i$ may switch directly to Section~\ref{secfirstproof}.

For all $0 \leq j < k \leq N$ and $U = (U_i )_{i=j,\ldots,k-1} \in (\R^m)^{k-j}$, we define the function $u_U : [s_j,s_{k}) \to \R^m$ by $u_U :=  \sum_{i=j}^{k-1} U_{i} \; \mathbf{1}_{[s_i,s_{i+1})}$. For every $y \in \R^n$, we denote by $q(\cdot,j,k,y,U) : [s_j,s_{k}] \rightarrow \R^n$ the unique absolutely continuous solution of the linear Cauchy problem
\begin{equation*}
\left\{\begin{split}
& \dot q(t) = A(t) q(t) + B(t) u_U(t) + \omega (t), \quad \text{a.e. $t \in [s_j,s_{k}]$,}  \\
& q(s_j)=y .
\end{split}\right.
\end{equation*}
For every $0 \leq j \leq N-1$, we define the nonnegative function $\overline{\mathcal{V}}_j (\cdot,\cdot) : \R^n \times (\R^m)^{N-j} \to \R^+$ by
\begin{multline}\label{eqdefvbar}
\overline{\mathcal{V}}_j ( y , U ) := \dfrac{1}{2} \Big\langle S \big( q(b,j,N,y,U)-q_b \big) , q(b,j,N,y,U)-q_b \Big\rangle_n \\[5pt]
+ \dfrac{1}{2} \di \int_{s_j}^{b} \Big\langle W(\t) \big( q(\t,j,N,y,U)-x(\t) \big) , q(\t,j,N,y,U) -x(\t) \Big\rangle_n \\
+ \Big\langle R(\t) \big( u_U(\t)-v(\t) \big) , u_U(\t)-v(\t) \Big\rangle_m \; d\t ,
\end{multline}
for every $(y,U) \in \R^n \times (\R^m)^{N-j}$.

\begin{remark}\label{remreplace}
Note that $\overline{\mathcal{V}}_j ( y , U )$ coincides with the cost $\CC(u_U)$ whenever the initial time $a$ is replaced by $s_j$ and the initial condition $q_a$ is replaced by $y$ and $h$ is replaced by $(h_i)_{i=j,\ldots,N-1}$ in Problem~($\mathscr{P}_{\E_h}$).
\end{remark}

In the sequel, we set
$$ \Phi_{j,y}(U_j) := q ( s_{j+1} , j , j+1, y , U_j ) \quad \text{and} \quad  \tilde{U} := (U_i)_{i=j+1,\ldots,N-1} , $$
for every $0 \leq j \leq N-1$ and every $(y,U) \in \R^n \times (\R^m)^{N-j}$. The next statement obviously follows from the definition of $\overline{\mathcal{V}}_j (\cdot,\cdot)$.

\begin{lemma}\label{lemlem1}
For every $0 \leq j \leq N-2$ and every $(y,U) \in \R^n \times (\R^m)^{N-j}$, we have
\begin{multline*}
\overline{\mathcal{V}}_j ( y , U ) = \overline{\mathcal{V}}_{j+1} \big( \Phi_{j,y}(U_j) , \tilde{U} \big) \\[7pt]
+ \dfrac{1}{2} \di \int_{s_j}^{s_{j+1}} \Big\langle W(\t) \big( q(\t,j,j+1,y, U_j )-x(\t) \big) , q(\t,j,j+1,y, U_j ) -x(\t) \Big\rangle_n \\
+ \Big\langle R(\t) \big( U_j -v(\t) \big) ,U_j -v(\t) \Big\rangle_m \; d\t .
\end{multline*}
\end{lemma}

Finally, for every $0 \leq j \leq N-1$, we define the so-called \textit{value function} $\mathcal{V}_j (\cdot)  : \R^n \to \R^+$ as the nonnegative function given by
\begin{equation}\label{eqdefVF}
\mathcal{V}_j (y) := \inf_{U \in (\R^m)^{N-j}} \overline{\mathcal{V}}_j (y,U),
\end{equation}
for every $y \in \R^n$. From Remark~\ref{remreplace} and similarly to Proposition~\ref{thmexistenceunicitepermanent}, one can easily prove that the infimum $\mathcal{V}_j(y)$ is reached at a unique point denoted by $U_j(y)^* = (U_j(y)^*_i)_{i=j,\ldots,N-1} \in (\R^m)^{N-j}$. As a consequence, we have
\begin{equation}\label{eq743}
\mathcal{V}_j (y) = \overline{\mathcal{V}}_j \big( y , U_j(y)^* \big),
\end{equation}
for every $0 \leq j \leq N-1$ and every $y \in \R^n$.

\begin{remark}\label{rem879}
From Remark~\ref{remreplace}, we have $  \CC(u^*_h) = \mathcal{V}_0(q_a) $ and $ (U^*_{h,i})_{i=0,\ldots,N-1} = U_0(q_a)^* $. More generally we have $ (U^*_{h,i})_{i=j,\ldots,N-1} = U_j( q(s_j,u^*_h) )^* $ for every $0 \leq j \leq N-1$.
\end{remark}

The following statement follows from the definition of $U_j(y)^*$.

\begin{lemma}\label{lemlem2}
For every $0 \leq  j \leq N-2$ and every $y \in \R^n$, we have
$$ U_{j+1} \Big( \Phi_{j,y} \big( U_j(y)^*_j \big) \Big)^* = \Big( U_j(y)^*_{i} \Big)_{i=j+1,\ldots,N-1} \in (\R^m)^{N-(j+1)}. $$
\end{lemma}

Finally, from \eqref{eq743}, Lemmas~\ref{lemlem1} and \ref{lemlem2}, we infer the following result.

\begin{proposition}\label{propfauxDPP}
For every $0 \leq  j \leq N-2$ and every $y \in \R^n$, we have
\begin{multline*}
\mathcal{V}_j (y) = \mathcal{V}_{j+1} \Big(  \Phi_{j,y} \big( U_j(y)^*_j \big) \Big) \\[7pt]
+  \dfrac{1}{2} \di \int_{s_j}^{s_{j+1}} \Big\langle W(\t) \big( q(\t,j,j+1,y, U_j(y)^*_j)-x(\t) \big) , q(\t,j,j+1,y, U_j(y)^*_j  ) -x(\t) \Big\rangle_n \\
+ \Big\langle R(\t) \big( U_j(y)^*_j  -v(\t) \big) ,U_j(y)^*_j  -v(\t) \Big\rangle_m \; d\t .
\end{multline*}
\end{proposition}

Proposition~\ref{propfauxDPP} will be used in order to prove the invertibility of the matrices $\TT_i$ in the first proof of Theorem~\ref{thmmainresult2} (based on Proposition~\ref{proppmpsample} and detailed in Section~\ref{secfirstproof}).

\begin{remark}
Proposition~\ref{propfauxDPP} does not correspond to the dynamic programming principle, whose version adapted to the framework of this paper is stated in Proposition~\ref{propDPP} (see Section~\ref{secsecondproof}).
\end{remark}

\subsection{Proof of Theorem~\ref{thmmainresult2} based on Proposition~\ref{proppmpsample}}\label{secfirstproof}

From Proposition~\ref{proppmpsample} we have
\begin{equation}\label{eq980}
\RR_i U^*_{h,i} = \RV_i + \int_{s_i}^{s_{i+1}} B(s)^\top p(s , u^*_h) \; ds,
\end{equation}
for all $i=0,\ldots,N-1$. In order to prove Theorem~\ref{thmmainresult2} using \eqref{eq980}, we use the Duhamel formula in order to derive an explicit expression for $p(s,u^*_h)$ as a function of $q(s_i,u^*_h)$ and $U^*_{h,i}$.
We will prove by backward induction that the following five statements are true:
\begin{enumerate}
\item $\TT_i$ is invertible;
\item $U^*_{h,i} = - {\TT}_i^{-1} ( \PP_i q (s_i , u^*_h) + \HH_i) $;
\item $p(s_i,u^*_h) = - ( \KK_{i} q(s_i,u^*_h) + \JJ_i )$;
\item $\mathcal{V}_i(y) = \frac{1}{2} \langle \KK_i y , y \rangle_n + \langle \JJ_i , y \rangle_n + \frac{1}{2} \YY_i$ for all $y \in \R^n$;
\item $\KK_{i}$ is positive-semidefinite;
\end{enumerate}
for every $i=N-1,\ldots,0$. As explained at the beginning of Section~\ref{secVF}, the reader who would like to skip, at least in a first step, technical difficulties related to the invertibility of the matrices $\TT_i$, may focus only on the statements 2 and 3 above.

To prove the induction steps, let us first recall the following equalities that follow from the  Duhamel formula:
\begin{equation}\label{eqduhamelp}
p(s,u^*_h) = Z(s_{i+1},s)^\top p(s_{i+1},u^*_h) - \int_{s}^{s_{i+1}} Z(\t,s)^\top W(\t) \big( q(\t, u^*_h)-x(\t) \big) \; d\t  ,
\end{equation}
\begin{equation}\label{eqduhamelq}
q(\t,u^*_h) = Z(\t,s_i) q(s_i,u^*_h) + \int_{s_i}^\t Z(\t,\xi) B(\xi) \; d\xi \; U^*_{h,i} + \int_{s_i}^\t Z(\t,\xi) \omega(\xi) \; d\xi ,
\end{equation}
for every $i=0,\ldots,N-1$ and for every $s$, $\t \in [s_i,s_{i+1}]$.

\paragraph{Initialization of the backward induction.}
Let $i=N-1$.

\textit{\textbf{1.}} Since $\KK_{i+1} = S$ is positive-semidefinite, we infer that $\TT_{i}$ is invertible (see Remark~\ref{remTiinv}).

\textit{\textbf{2.}} Using \eqref{eqduhamelp} and \eqref{eqduhamelq} and $p(s_{i+1},u^*_h) = p(b,u^*_h) = -S (q(b,u^*_h) - q_b) = -S (q(s_{i+1},u^*_h) - q_b) $ (see Lemma~\ref{propcauchyp}), we get that
\begin{multline}\label{eq392}
p(s,u^*_h) = - \Big( Z({s_{i+1}},s)^\top S Z({s_{i+1}},s_i) + \int_s^{s_{i+1}} Z(\t,s)^\top W(\t) Z(\t,s_i) \; d\t \Big) q(s_i,u^*_h) \\
- \Big( Z({s_{i+1}},s)^\top S \ZB_i + \int_s^{s_{i+1}} Z(\t,s)^\top W(\t) \int_{s_i}^\t Z(\t,\xi) B(\xi) \; d\xi \; d\t \Big) U^*_{h,i} \\
- \Big( Z({s_{i+1}},s)^\top S \ZO_i + \int_s^{s_{i+1}} Z(\t,s)^\top W(\t) \left( \int_{s_i}^\t Z(\t,\xi) \omega(\xi) \; d\xi - x(\t) \right) \; d\t \Big).
\end{multline}
Replacing $p(s,u^*_h)$ in Equality~\eqref{eq980} and applying the Fubini theorem, we obtain that
\begin{multline*}
\RR_i U^*_{h,i} = \RV_i - \Big( \ZB_i^\top S Z(s_{i+1},s_i) + \ZBWZ_i \Big) q(s_i,u^*_h) \\ - \Big( \ZB_i^\top S \ZB_i + \ZBWZB_i \Big)U^*_{h,i} - \Big( \ZB_i^\top S \ZO_i + \ZBWZOX_i \Big),
\end{multline*}
that is exactly $ \TT_i U^*_{h,i} = - \PP_i q (s_i , u^*_h) - \HH_i $. Since $\TT_i$ is invertible, we infer that
$$ U^*_{h,i} = - {\TT}_i^{-1} \big( \PP_i q (s_i , u^*_h) + \HH_i \big) .$$

\textit{\textbf{3.}} Taking $s=s_i$ in \eqref{eq392} leads to
\begin{multline*}
p(s_i,u^*_h) = - \Big( Z({s_{i+1}},s_i)^\top S Z({s_{i+1}},s_i) + \ZWZ_i \Big) q(s_i,u^*_h) \\
- \Big( Z({s_{i+1}},s_i)^\top S \ZB_i + \ZBWZ_i^\top \Big) U^*_{h,i} - \Big( Z({s_{i+1}},s_i)^\top S \ZO_i + \ZWZOX_i \Big),
\end{multline*}
that is exactly $p(s_i,u^*_h) = - \QQ_i q(s_i,u^*_h) - \PP_i^\top U^*_{h,i} - \GG_i$. Since $U^*_{h,i} = - {\TT}_i^{-1} ( \PP_i q (s_i , u^*_h) + \HH_i )$, we conclude that
$$ p(s_i,u^*_h) = - \big( \KK_{i} q(s_i,u^*_h) + \JJ_i \big). $$

\textit{\textbf{4.}} Let $y \in \R^n$. From Remarks~\ref{termdependant} and \ref{remreplace} and from the definition of $U_i(y)^*$, similarly to Step 2, we get that $U_i(y)^* = - {\TT}_i^{-1} ( \PP_i y + \HH_i ) \in \R^m$. Besides, it follows from the Duhamel formula that
$$ q(\t,i,i+1,y,U_i(y)^*) = Z(\t,s_i)y + \int_{s_i}^\t Z(\t,\xi) B(\xi) \; d\xi \; U_i(y)^* + \int_{s_i}^\t Z(\t,\xi) \omega(\xi) \; d\xi, $$
for every $\t \in [s_i,b]$. Using the above equality in \eqref{eqdefvbar}, we exactly obtain that
$$ \mathcal{V}_i(y) = \overline{\mathcal{V}}_i (y,U_i(y)^*) = \left\langle \frac{1}{2} \TT_i U_i(y)^* + \PP_i y + \HH_i , U_i(y)^* \right\rangle_m + \left\langle \frac{1}{2} \QQ_i y + \GG_i , y \right\rangle_n +  \frac{1}{2} \FF_i. $$
Since $U_i(y)^* = - {\TT}_i^{-1} ( \PP_i y + \HH_i )$, we conclude that
$$ \mathcal{V}_i(y) = \frac{1}{2} \langle \KK_i y , y \rangle_n + \langle \JJ_i , y \rangle_n + \frac{1}{2} \YY_i .$$

\textit{\textbf{5.}} Let $y \in \R^n$ and let us consider temporarily the homogeneous Problem~$(\mathscr{P}_{\E_h})$, that is, let us consider temporarily that $q_b = x(t) = \omega (t) = 0_{\R^n}$ and $v(t)= 0_{\R^m}$ for every $t \in [a,b]$. From Remarks~\ref{cashomogene} and \ref{Kindependant}, similarly to Step 4, we get, in the homogeneous case, that $ \mathcal{V}_i(y) = \frac{1}{2} \langle \KK_i y , y \rangle_n \geq 0$. It follows that $\KK_i$ is positive-semidefinite.

\paragraph{The induction step.} Let $i \in \{ 0, \ldots ,N-2 \}$ and let us assume that the five statements are satisfied at steps $i+1, \ldots ,N-1$.

\textit{\textbf{1.}} Since $\KK_{i+1}$ is positive-semidefinite, we infer that $\TT_{i}$ is invertible (see Remark~\ref{remTiinv}).

\textit{\textbf{2.}} Using \eqref{eqduhamelp} and \eqref{eqduhamelq} and $p(s_{i+1},u^*_h) = - ( \KK_{i+1} q( s_{i+1} , u^*_h) + \JJ_{i+1} )$, we get that
\begin{multline}\label{eq392bis}
p(s,u^*_h) = - \Big( Z(s_{i+1},s)^\top \KK_{i+1} Z(s_{i+1},s_i) + \int_s^{s_{i+1}} Z(\t,s)^\top W(\t) Z(\t,s_i) \; d\t \Big) q(s_i,u^*_h) \\
- \Big( Z(s_{i+1},s)^\top \KK_{i+1} \ZB_i + \int_s^{s_{i+1}} Z(\t,s)^\top W(\t) \int_{s_i}^\t Z(\t,\xi) B(\xi) \; d\xi \; d\t \Big) U^*_{h,i} \\
- \Big( Z(s_{i+1},s)^\top \KK_{i+1} \ZO_i \\ + \int_s^{s_{i+1}} Z(\t,s)^\top W(\t) \left( \int_{s_i}^\t Z(\t,\xi) \omega(\xi) \; d\xi - x(\t) \right) \; d\t + Z(s_{i+1},s)^\top \JJ_{i+1} \Big).
\end{multline}
Replacing $p(s,u^*_h)$ in \eqref{eq980} and applying the Fubini theorem, we obtain that
\begin{multline*}
\RR_i U^*_{h,i} = \RV_i - \Big( \ZB_i^\top \KK_{i+1} Z(s_{i+1},s_i) + \ZBWZ_i \Big) q(s_i,u^*_h) \\ - \Big( \ZB_i^\top \KK_{i+1} \ZB_i + \ZBWZB_i \Big)U^*_{h,i} - \Big( \ZB_i^\top \KK_{i+1} \ZO_i + \ZBWZOX_i + \ZB_i^\top \JJ_{i+1} \Big),
\end{multline*}
that is exactly $ \TT_i U^*_{h,i} = - \PP_i q (s_i , u^*_h) - \HH_i $. Since $\TT_i$ is invertible, we infer that
$$U^*_{h,i} = - {\TT}_i^{-1} \big( \PP_i q (s_i , u^*_h) + \HH_i \big).$$

\textit{\textbf{3.}} Taking $s=s_i$ in \eqref{eq392bis} leads to
\begin{multline*}
p(s_i,u^*_h) = - \Big( Z(s_{i+1},s_i)^\top \KK_{i+1} Z(s_{i+1},s_i) + \ZWZ_i \Big) q(s_i,u^*_h) \\
- \Big( Z(s_{i+1},s_i)^\top \KK_{i+1} \ZB_i + \ZBWZ_i^\top \Big) U^*_{h,i} \\ - \Big( Z(s_{i+1},s_i)^\top \KK_{i+1} \ZO_i + \ZWZOX_i + Z(s_{i+1},s_i)^\top \JJ_{i+1} \Big),
\end{multline*}
that is exactly $p(s_i,u^*_h) = - \QQ_i q(s_i,u^*_h) - \PP_i^\top U^*_{h,i} - \GG_i$. Since $U^*_{h,i} = - {\TT}_i^{-1} ( \PP_i q (s_i , u^*_h) + \HH_i )$, we infer that
\begin{equation*}
p(s_i,u^*_h) = - \big( \KK_{i} q(s_i,u^*_h) + \JJ_i \big).
\end{equation*}

\textit{\textbf{4.}} Let $y \in \R^n$. From Remarks~\ref{termdependant} and \ref{remreplace} and from the definition of $U_i(y)^*$, one can prove in a very similar way than Step 2. that $U_i(y)^*_i = - {\TT}_i^{-1} ( \PP_i y + \HH_i ) \in \R^m$. From Proposition~\ref{propfauxDPP}, we have
\begin{multline}\label{eq401}
\mathcal{V}_i (y) = \mathcal{V}_{i+1} \Big(  \Phi_{i,y} \big( U_i(y)^*_i \big) \Big) \\[7pt]
+  \dfrac{1}{2} \di \int_{s_i}^{s_{i+1}} \Big\langle W(\t) \big( q(\t,i,i+1,y, U_i(y)^*_i)-x(\t) \big) , q(\t,i,i+1,y, U_i(y)^*_i  ) -x(\t) \Big\rangle_n \\
+ \Big\langle R(\t) \big( U_i(y)^*_i  -v(\t) \big) ,U_i(y)^*_i  -v(\t) \Big\rangle_m \; d\t .
\end{multline}
From the induction hypothesis, we have
\begin{multline}\label{eq402}
\mathcal{V}_{i+1} \Big(  \Phi_{i,y} \big( U_i(y)^*_i \big) \Big)= \frac{1}{2} \Big\langle \KK_{i+1}  \Phi_{i,y} \big( U_i(y)^*_i \big) ,  \Phi_{i,y} \big( U_i(y)^*_i \big) \Big\rangle_n \\
+ \Big\langle \JJ_{i+1} ,  \Phi_{i,y} \big( U_i(y)^*_i \big) \Big\rangle_n + \frac{1}{2} \YY_{i+1} .
\end{multline}
On the other hand, it follows from the Duhamel formula that
\begin{equation}\label{eq403}
q(\t,i,i+1,y,U_i(y)^*_i) = Z(\t,s_i) y + \int_{s_i}^\t Z(\t,\xi) B(\xi) \; d\xi \; U_i(y)^*_i + \int_{s_i}^\t Z(\t,\xi) \omega(\xi) \; d\xi,
\end{equation}
for every $\t \in [s_i,s_{i+1}]$. Taking $\t = s_{i+1}$ in \eqref{eq403}, we get that
\begin{equation}\label{eq404}
\Phi_{i,y} \big( U_i(y)^*_i \big) = Z(s_{i+1},s_i) y + \ZB_i U_i(y)^*_i + \ZO_i .
\end{equation}
Finally, using \eqref{eq402}, \eqref{eq403} and \eqref{eq404} in \eqref{eq401} yields
$$ \mathcal{V}(i,y) = \left\langle \frac{1}{2} \TT_i U_i(y)^*_i + \PP_i y + \HH_i , U_i(y)^*_i \right\rangle_m + \left\langle \frac{1}{2} \QQ_i y + \GG_i , y \right\rangle_n +  \frac{1}{2} \FF_i. $$
Since $U_i(y)^*_i = - {\TT}_i^{-1} ( \PP_i y + \HH_i )$, we conclude that
$$ \mathcal{V}(i,y) = \frac{1}{2} \langle \KK_i y , y \rangle_n + \langle \JJ_i , y \rangle_n + \frac{1}{2} \YY_i .$$

\textit{\textbf{5.}} Let $y \in \R^n$ and let us consider temporarily the homogeneous Problem~$(\mathscr{P}_{\E_h})$, that is, let us consider temporarily that $q_b = x(t) = \omega (t) = 0_{\R^n}$ and $v(t)= 0_{\R^m}$ for every $t \in [a,b]$. From Remarks~\ref{cashomogene} and \ref{Kindependant}, similarly to Step 4, we get, in the homogeneous case, that $ \mathcal{V}_i(y) = \frac{1}{2} \langle \KK_i y , y \rangle_n \geq 0$. It follows that $\KK_i$ is positive-semidefinite.

\subsection{Proof of Theorem~\ref{thmmainresult2} based on the dynamic programming principle}\label{secsecondproof}

This section is dedicated to an alternative proof of Theorem~\ref{thmmainresult2}, based on the following dynamic programming principle.

\begin{proposition}[Dynamic programming principle]\label{propDPP}
For every $0 \leq j \leq N-2$ and every $y \in \R^n$, we have
\begin{multline*}
\mathcal{V}_j(y) = \inf_{U_j \in \R^m} \Big(  \mathcal{V}_{j+1} \big( \Phi_{j,y}(U_j) \big) \\
\qquad\qquad\qquad + \dfrac{1}{2} \di \int_{s_j}^{s_{j+1}} \Big( \Big\langle W(\t) \big( q(\t,j,j+1,y,U_j)-x(\t) \big) , q(\t,j,j+1,y,U_j) -x(\t) \Big\rangle_n \\
+ \Big\langle R(\t) \big( U_j-v(\t) \big) , U_j-v(\t) \Big\rangle_m \Big) d\t  \Big).
\end{multline*}
\end{proposition}

\begin{proof}
Let $0 \leq j \leq N-2$ and $y \in \R^n$. From \eqref{eqdefVF} and Lemma~\ref{lemlem1}, it is clear that $\mathcal{V}_j(y)$ is equal to the infimum of
\begin{multline*}
\overline{\mathcal{V}}_{j+1} \big( \Phi_{j,y}(U_j) , \tilde{U} \big) 
+ \dfrac{1}{2} \di \int_{s_j}^{s_{j+1}} \Big( \Big\langle W(\t) \big( q(\t,j,j+1,y, U_j )-x(\t) \big) , q(\t,j,j+1,y, U_j ) -x(\t) \Big\rangle_n \\
+ \Big\langle R(\t) \big( U_j -v(\t) \big) ,U_j -v(\t) \Big\rangle_m \Big) d\t ,
\end{multline*}
over all possible $U = (U_j,\tilde{U}) \in \R^m \times (\R^m)^{N-(j+1)}$. Using that
\begin{equation}\label{eq8795}
\inf_{(\mu_1,\mu_2) \in \Gamma_1 \times \Gamma_2} \Psi (\mu_1,\mu_2) = \inf_{\mu_1 \in \Gamma_1} \left( \inf_{\mu_2 \in \Gamma_2} \Psi (\mu_1,\mu_2) \right) ,
\end{equation}
for any function $\Psi : \Gamma_1 \times \Gamma_2 \rightarrow \R$ and any couple $(\Gamma_1,\Gamma_2)$ of nonempty sets, we conclude the proof by applying \eqref{eq8795} to $\mathcal{V}_j(y)$ (separating the variables $U_j$ and $\tilde{U}$).
\end{proof}

In order to prove Theorem~\ref{thmmainresult2}, we will prove by backward induction that the following four statements are true:
\begin{enumerate}
\item $\TT_i$ is invertible;
\item $U^*_{h,i} = - {\TT}_i^{-1} ( \PP_i q (s_i , u^*_h) + \HH_i) $;
\item $\mathcal{V}(i,y) = \frac{1}{2} \langle \KK_i y , y \rangle_n + \langle \JJ_i , y \rangle_n + \frac{1}{2} \YY_i$ for all $y \in \R^n$;
\item $\KK_{i}$ is positive-semidefinite;
\end{enumerate}
for every $i=N-1,\ldots,0$.

Let us first recall the following equality that follows from the Duhamel formula:
\begin{equation}\label{eqduhamelq2}
q(\t,i,i+1,y,U_i) = Z(\t,s_i) y + \int_{s_i}^\t Z(\t,\xi) B(\xi) \; d\xi \; U_i + \int_{s_i}^\t Z(\t,\xi) \omega(\xi) \; d\xi ,
\end{equation}
for every $0 \leq i \leq N-1$, every $(y,U_i) \in \R^n \times \R^m$ and every $\t \in [s_i,s_{i+1}]$. Taking $0 \leq i \leq N-2$ and $\t = s_{i+1}$ in \eqref{eqduhamelq2}, we get that
\begin{equation}\label{eqduhamelq22}
\Phi_{i,y}(U_i) = Z(s_{i+1},s_i) y + \ZB_i U_i + \ZO_i
\end{equation}

\paragraph{Initialization of the backward induction.}
Let $i=N-1$.

\textit{\textbf{1.}} Since $\KK_{i+1} = S$ is positive-semidefinite, we infer that $\TT_{i}$ is invertible (see Remark~\ref{remTiinv}).

\textit{\textbf{2.}} Let $y \in \R^n$. Using \eqref{eqduhamelq2} in \eqref{eqdefvbar}, we get that
\begin{multline*}
\overline{\mathcal{V}}_i(y,U_i) = \dfrac{1}{2} \left\langle  \Big( \ZB_i^\top S \ZB_i + \ZBWZB_i + \RR_i \Big)U_i , U_i\right\rangle_m \\
+ \left\langle  \Big( \ZB_i^\top {S} Z(s_{i+1},s_i) + \ZBWZ_i \Big)y + \Big( \ZB_i^\top {S} \ZO_i + \ZBWZOX_i - \RV_i \Big) , U_i \right\rangle_m \\
+ \left\langle \dfrac{1}{2} \Big( Z(s_{i+1},s_i)^\top S Z(s_{i+1},s_i) + \ZWZ_i \Big) y + \Big( Z(s_{i+1},s_i)^\top {S} \ZO_i + \ZWZOX_i \Big) , y \right\rangle_n \\
+ \dfrac{1}{2} \Big( \langle S \ZO_i ,  \ZO_i \rangle_n + \WZOX_i^2 + \RV_i^2 \Big) ,
\end{multline*}
for every $U_i \in \R^m$. Hence we have exactly obtained that
\begin{equation}\label{eq098}
\mathcal{V}_i(y) = \inf_{U_i \in \R^m} \overline{\mathcal{V}}_i(y,U_i)  = \inf_{U_i \in \R^m} \left( \left\langle \frac{1}{2} \TT_i U_i + \PP_i y + \HH_i , U_i \right\rangle_m + \left\langle \frac{1}{2} \QQ_i y + \GG_i , y \right\rangle_n +  \frac{1}{2} \FF_i \right).
\end{equation}
Differentiating the above expression with respect to $U_i$, the infimum $\mathcal{V}_i(y)$ is reached at $U_i(y)^* \in \R^m$ that satisfies $ {\TT}_i U_i(y)^* + \PP_i y + \HH_i = 0_{\R^m}$. Since $\TT_i$ is invertible, we deduce that $ U _i(y)^* = - {\TT}_i^{-1} ( \PP_i y + \HH_i) $. Finally, taking $y = q(s_i,u^*_h)$, we obtain from Remark~\ref{rem879} that
$$ U^*_{h,i} =  U_i \big( q(s_i,u^*_h) \big)^* = - {\TT}_i^{-1} \big( \PP_i q(s_i,u^*_h) + \HH_i \big) .$$

\textit{\textbf{3.}} Let $y \in \R^n$. Replacing $U_i$ by $U_i(y)^* = - {\TT}_i^{-1} ( \PP_i y + \HH_i)$ in the expression of~\eqref{eq098}, we get that
$$ \mathcal{V}_i(y)  = \dfrac{1}{2} \left\langle \KK_i y , y \right\rangle_n + \left\langle \JJ_i , y \right\rangle_n + \dfrac{1}{2} \YY_i .$$

\textit{\textbf{4.}} Let $y \in \R^n$ and let us consider temporarily the homogeneous Problem~$(\mathscr{P}_{\E_h})$, that is, let us consider temporarily that $q_b = x(t) = \omega (t) = 0_{\R^n}$ and $v(t)= 0_{\R^m}$ for every $t \in [a,b]$. From Remarks~\ref{cashomogene} and \ref{Kindependant}, similarly to Step 3, we get, in the homogeneous case, that $ \mathcal{V}_i(y) = \frac{1}{2} \langle \KK_i y , y \rangle_n \geq 0$. It follows that $\KK_i$ is positive-semidefinite.

\paragraph{Induction step.}
Let $i \in \{ 0, \ldots ,N-2 \}$ and let us assume that the four statements are satisfied at steps $i+1, \ldots ,N-1$.

\textit{\textbf{1.}} Since $\KK_{i+1}$ is positive-semidefinite, we infer that $\TT_{i}$ is invertible (see Remark~\ref{remTiinv}).

\textit{\textbf{2.}} Let $y \in \R^n$. From the dynamic programming principle stated in Proposition~\ref{propDPP}, we obtain that $\mathcal{V}_i(y)$ is equal to the infimum of
\begin{multline}\label{eq099}
\mathcal{V}_{i+1} \big( \Phi_{i,y}(U_i) \big)
+ \dfrac{1}{2} \di \int_{s_i}^{s_{i+1}} \Big( \Big\langle W(\t) \big( q(\t,i,i+1,y,U_i)-x(\t) \big) , q(\t,i,i+1,y,U_i) -x(\t) \Big\rangle_n \\
+ \Big\langle R(\t) \big( U_i-v(\t) \big) , U_i-v(\t) \Big\rangle_m \Big) d\t
\end{multline}
over $U_i \in \R^m$. From the induction hypothesis, we have
\begin{equation}\label{eq199}
\mathcal{V}_{i+1} \big( \Phi_{i,y}(U_i) \big)  =  \dfrac{1}{2} \big\langle \KK_{i+1} \Phi_{i,y}(U_i) , \Phi_{i,y}(U_i) \big\rangle_n + \big\langle \JJ_{i+1} , \Phi_{i,y}(U_i) \big\rangle_n + \dfrac{1}{2} \YY_{i+1} .
\end{equation}
Using \eqref{eqduhamelq2}, \eqref{eqduhamelq22} and \eqref{eq199} in \eqref{eq099}, we obtain that
\begin{equation}\label{eq299}
\mathcal{V}_i(y) = \inf_{U_i \in \R^m} \left( \left\langle \frac{1}{2} \TT_i U_i + \PP_i y + \HH_i , U_i \right\rangle_m + \left\langle \frac{1}{2} \QQ_i y + \GG_i , y \right\rangle_n +  \frac{1}{2} \FF_i \right).
\end{equation}
Differentiating the above expression with respect to $U_i$, the infimum $\mathcal{V}_i(y)$ is reached in $U_i(y)^*_i \in \R^m$ that satisfies $ {\TT}_i U_i(y)^*_i + \PP_i y + \HH_i = 0_{\R^m}$. Since $\TT_i$ is invertible, we infer that $ U_i(y)^*_i = - {\TT}_i^{-1} ( \PP_i y + \HH_i) $. Finally, taking $y = q(s_i,u^*_h)$, we obtain from Remark~\ref{rem879} that
$$ U^*_{h,i} =  U_i \big( q(s_i,u^*_h) \big)^*_i = - {\TT}_i^{-1} \big( \PP_i q(s_i,u^*_h) + \HH_i \big)  .$$

\textit{\textbf{3.}} Let $y \in \R^n$. Replacing $U_i$ by $U_i(y)^*_i = - {\TT}_i^{-1} ( \PP_i y + \HH_i)$ in \eqref{eq299}, we get that
$$ \mathcal{V}_i(y)  = \dfrac{1}{2} \left\langle \KK_i y , y \right\rangle_n + \left\langle \JJ_i , y \right\rangle_n + \dfrac{1}{2} \YY_i .$$

\textit{\textbf{4.}} Let $y \in \R^n$ and let us consider temporarily the homogeneous Problem~$(\mathscr{P}_{\E_h})$, that is, let us consider temporarily that $q_b = x(t) = \omega (t) = 0_{\R^n}$ and $v(t)= 0_{\R^m}$ for every $t \in [a,b]$. From Remarks~\ref{cashomogene} and \ref{Kindependant}, similarly to Step 3, we get, in the homogeneous case, that $ \mathcal{V}_i(y) = \frac{1}{2} \langle \KK_i y , y \rangle_n \geq 0$. It follows that $\KK_i$ is positive-semidefinite.

\section{Conclusion}
We have extended the Riccati theory for general linear-quadratic optimal control problems with sampled-data controls, by two approaches. The first approach consists of applying an appropriate version of the Pontryagin maximum principle, which is adapted to optimal sampled-data control problems, and of showing that the costate can be expressed linearly in function of the state, by introducing an adequate version of the Riccati equation. The second approach relies on an appropriate version of the dynamic programming principle, combined with backward induction arguments.

We have also proved that the optimal sampled-data controls converge pointwise to the optimal permanent control as the sampling periods tend to zero.


As an open problem, it is natural to raise these questions for more general optimal control problems, having nonlinear dynamics and possibly involving constraints on the final state.

When dealing with nonlinear dynamics, the Riccati equation becomes the Hamilton-Jacobi equation, which is a first-order partial differential equation of which viscosity solutions are nonsmooth in general. A first open issue would be to investigate the dynamic programming principle, and the Hamilton-Jacobi equation, in such a nonlinear context with sampled-data controls.

When considering constraints on the final state, even in the linear-quadratic case the situation is more involved, and establishing a convergence result like Theorem \ref{thmmainresult} may already be challenging and will certainly require to consider finer concepts like singular trajectories or abnormal extremals, and conjugate point theory (see \cite{agrach,Bon-Chy03}).

\bigskip

\noindent{\bf Acknowledgment.}
The second author was partially supported by the Grant FA9550-14-1-0214 of the EOARD-AFOSR.

\end{document}